\RequirePackage{etoolbox}
\csdef{input@path}{{style/}{graphics/}}
\documentclass[aos,MSNbibl,seceqn,dvips]{arximspdf}
\usepackage{dcolumn}
\usepackage{algorithmicx}
\usepackage{algpseudocode}
\usepackage{algorithm}
\usepackage{graphicx}

\doi{10.1214/14-AOS1282}
\volume{43}
\issue{1}
\pubyear{2015}
\firstpage{422}
\lastpage{461}
\docsubty{FLA}

\makeatletter
\newcolumntype{d}[1]{D{.}{.}{#1}}

\newcommand{\lleft}{\left}
\newcommand{\rright}{\right}
\renewcommand{\mid}{|}
\newcommand{\eqref}[1]{(\ref{#1})}
\newtheorem{theorem}{Theorem}[section]
\newtheorem{proposition}[theorem]{Proposition}
\newproclaim{example}[theorem]{Example}
\newtheorem{lemma}[theorem]{Lemma}
\newtheorem{corollary}[theorem]{Corollary}
\newproclaim{definition}[theorem]{Definition}
\newproclaim{remark}[theorem]{Remark}
\newcommand{\perpp}{\perp\!\!\! \perp}
\newcommand{\PP}{\mathbb{P}}
\newcommand{\R}{\mathbb{R}}
\newcommand{\Q}{\mathbb{Q}}
\newcommand{\cone}{\operatorname{cone}}
\newcommand{\conv}{\operatorname{conv}}
\newcommand{\rank}{\operatorname{rank}}
\makeatother

\begin{document}
\begin{frontmatter}

\title{Fixed points of the EM algorithm and nonnegative rank
boundaries}
\runtitle{Fixed points of EM and nonnegative rank}

\begin{aug}
\author[A]{\fnms{Kaie}~\snm{Kubjas}\corref{}\ead[label=e1]{kaie.kubjas@gmail.com}},
\author[B]{\fnms{Elina}~\snm{Robeva}\thanksref{T2}\ead[label=e2]{erobeva@gmail.com}}
\and
\author[B]{\fnms{Bernd}~\snm{Sturmfels}\thanksref{T3}\ead[label=e3]{bernd@berkeley.edu}}
\runauthor{K. Kubjas, E. Robeva and B. Sturmfels}
\affiliation{Aalto University, University of California, Berkeley\\
and University of California, Berkeley}
\address[A]{K. Kubjas\\
Aalto Science Institute\\
Aalto University\\
PO Box 15500\\
FI-00076 Aalto\\
Finland\\
\printead{e1}}
\address[B]{E. Robeva\\
B. Sturmfels\\
Department of Mathematics\\
University of California, Berkeley\\
Berkeley, California 94720\\
USA\\
\printead{e2}\\
\phantom{E-mail: }\printead*{e3}}
\end{aug}
%
\thankstext{T2}{Supported by a UC Berkeley Graduate Fellowship.}
\thankstext{T3}{Supported by NSF Grant DMS-09-68882.}

%
\received{\smonth{12} \syear{2013}}
%
\revised{\smonth{10} \syear{2014}}

%
\begin{abstract}
Mixtures of $r$ independent distributions
for two discrete random variables can be represented by matrices of
nonnegative rank $r$.
Likelihood inference for the model of such joint distributions leads to
problems in real algebraic geometry that are
addressed here for the first time.
We characterize the set of fixed points of the
Expectation--Maximization algorithm, and we study the
boundary of the space of matrices with nonnegative rank at most~$3$.
Both of these sets correspond to algebraic varieties with many
irreducible components.
\end{abstract}

%
\end{frontmatter}

\section{Introduction} \label{intro}
The $r$th mixture model $\mathcal{M}$ of two discrete random variables
$X $ and $Y$
expresses the conditional independence statement
$X \perpp Y \mid Z$,
where $Z$ is a hidden (or latent) variable with $r$ states.
Assuming that $X$ and $Y$ have $m$ and $n$ states, respectively,
their joint distribution is written as an $m \times n$-matrix
of nonnegative rank $\leq r$ whose
entries sum to $1$.
This mixture model is also known as the
\textit{naive Bayes model}. Its
graphical representation is shown
in Figure~\ref{figuregraphicalmodel}.

A collection of i.i.d.~samples from a joint distribution is recorded in a
nonnegative matrix
\[
U    =    \lleft[\matrix{ u_{11} & u_{12} & \cdots&
u_{1n}
\cr
u_{21} & u_{22} & \cdots&
u_{2n}
\cr
\vdots& \vdots& \ddots& \vdots
\cr
u_{m1} &
u_{m2} & \cdots& u_{mn}} \rright]. %
\]
Here, $u_{ij}$ is the number of observations in the sample with $X = i$
and $Y=j$.
The sample size is $u_{++} = \sum_{i,j} u_{ij}$.
It is standard practice to fit the model to the data $U$
using the Expectation--Maximization (EM) algorithm. However, it has been
pointed out in the literature that EM has several issues (see the next
paragraph for details) and one has to be careful when using it.
Our goal is to better understand this algorithm by studying
its mathematical properties in some detail.

One of the main issues of Expectation--Maximization is that it does not
provide a certificate for having found the global optimum.
The geometry of the algorithm has been a topic
for debate among statisticians since the seminal paper of
Dempster, Laird and Rubin \cite{DLR}.
Murray \cite{Murray} responded with
a warning for practitioners to
be aware of the existence of multiple stationary points.
Beale \cite{Beale} also brought this up, and
Fienberg \cite{Fienberg} referred to the
possibility that the MLE lies on the boundary
of the parameter space. A recent discussion of this issue
was presented by
Zwiernik and Smith \cite{ZS}, Section~3,
in their analysis of inferential
problems arising from the semialgebraic geometry
of a latent class model.
The fact that our model fails to be identifiable
was highlighted by Fienberg et al.
in \cite{FHRZ}, Section~4.2.3. This
poses additional difficulties, and it forces
us to distinguish between the boundary of the
parameter space and the
boundary of the model.
The image of the former contains the latter.

The EM algorithm aims to maximize the log-likelihood
function of the model~$\mathcal{M}$. In doing so, it approximates
the data matrix
$U $ with a product of nonnegative matrices $ A \cdot B$
where $A$ has $r$ columns and $B$ has $r$ rows.
In Section~\ref{sec2}, we review the EM algorithm in our context.
Here, it is essentially equivalent
to the widely used method of Lee and Seung \cite{LS}
for \textit{nonnegative matrix factorization}. The
nonnegative rank of matrices has been studied
from a broad range of perspectives, including
computational geometry \cite{ABRS,CR},
topology \cite{MSS}, contingency tables \cite{BCR,FHRZ},
complexity theory \cite{Moi,Vav}
and convex optimization \cite{FH}.
We here present the approach from
algebraic statistics \cite{LiAS,ASCB}.

\begin{figure}

\includegraphics{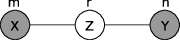}

\caption{Graphical model on two observed variables and one hidden variable.}
\label{figuregraphicalmodel}
\end{figure}

Maximum likelihood estimation for the model $\mathcal{M}$
is a nonconvex optimization problem. 
Any algorithm that promises to compute the MLE $\widehat P$
will face the following fundamental dichotomy. The
optimal matrix $\widehat P$ either
lies in the relative interior of $\mathcal{M}$
or it lies in the model boundary $\partial\mathcal{M}$.

If $\widehat P$ lies in the relative interior of $\mathcal M$, then
the situation is nice. In this case,
$\widehat P$ is a critical point for the likelihood function on the
manifold of rank $r$ matrices.
There are methods by Hauenstein et al. \cite{HRS} for finding
the MLE with certificate.
The ML degree, which they compute, bounds the number
of critical points, and hence all candidates for the global
maximizer $\widehat P$. However, things are more difficult when
$\widehat P$ lies in
the boundary $\partial\mathcal{M}$. In that case, $\widehat P$
is generally not a critical point for the likelihood function in the
manifold of rank $r$ matrices,
and none of the results on ML degrees in
\cite{LiAS,FHRZ,GR,HRS,HS} are\vspace*{1pt} applicable.
The present paper is the first to address the question of how $\widehat P$
varies when it occurs in the boundary $\partial\mathcal{M}$.
Table~\ref{tablenotcritical} underscores the significance of our approach.
As the matrix size grows, the boundary case is much
more likely to happen for randomly chosen input $U$.
The details for choosing $U$ and the simulation study that generated
Table~\ref{tablenotcritical} will be described in
Example~\ref{remtable1experiment}.

\begin{table}
\tabcolsep=0pt
\tablewidth=250pt
\caption{Percentage of data matrices whose maximum likelihood
estimate $\widehat P$ lies in the boundary $\partial\mathcal{M}$}
\label{tablenotcritical}
\begin{tabular*}{\tablewidth}{@{\extracolsep{\fill}}@{}lccccc@{}}
\hline
& \multicolumn{5}{c@{}}{\textbf{Size}}\\[-5pt]
& \multicolumn{5}{c@{}}{\hrulefill}\\
\textbf{Rank} & $\bolds{4\times4}$& $\bolds{5\times5}$ & $\bolds{6\times6}$ & $\bolds{7\times7}$& $\bolds{8\times8}$\\
\hline
3 & $4.4\%$ & 23\% & 49\% & 62\%& 85\% \\
4 & & \phantom{0}7\% &37\% &71\% & 95\%\\
5 & & & 10\%& 55\% & 96\%\\
6 & && & 20\% & 75\% \\
7 & & & & & 24\%\\
\hline
\end{tabular*}
\end{table}

We now summarize the contents of this article.
Section~\ref{sec1} furnishes an introduction to the geometry
of the mixture model $\mathcal{M}$
from Figure~\ref{figuregraphicalmodel}.
We define the
\textit{topological boundary} of $\mathcal{M}$ and the \textit{algebraic
boundary} of $\mathcal{M}$, and we explain how these
two notions of boundary differ.
Concrete numerical examples for
\mbox{$4 \times4$-}matrices of rank $3$ demonstrate
how $\widehat P$ behaves as the data $U$ vary.

In Section~\ref{sec2}, we review the EM algorithm for the model
$\mathcal{M}$,
and we identify its fixed points in the parameter space.
The main result is the characterization
of the set of fixed points in Theorem~\ref{thmEMfixedTheta}.

In Section~\ref{sec3}, we identify $\mathcal{M}$
with the set of
matrices of nonnegative rank at most $3$.
Theorem~\ref{theoremsemialgebraicdescription} gives a
quantifier-free formula
for this semialgebraic set.
The importance of finding such a formula
was already stressed in the articles \cite{ARSZ,ART}.
The resulting membership test for $\mathcal{M}$
is very fast and can be applied to matrices that contain parameters.
The proof of Theorem~\ref{theoremsemialgebraicdescription}
is based on the familiar characterization of
nonnegative rank in terms of nested polytopes \cite{ABRS,CR,Vav},
and, in particular, on work of Mond et al. \cite{MSS}
on the structure of
critical configurations in the plane (shown in Figure~\ref{figcriticalconfigurations}).

In Section~\ref{sec4}, we return to Expectation--Maximization, and we
study the system of equations that characterize the EM fixed points.
Proposition~\ref{propquivercycle}
characterizes its solutions in the interior of $\mathcal{M}$.
Even in the smallest interesting case, $m=n=4$ and $r=3$,
the variety of all EM fixed points has a huge number of irreducible components,
to be determined and interpreted in Theorem~\ref{thmem44}.

The most interesting among these are
the $288$ components that delineate the topological boundary $\partial
\mathcal{M}$
inside the simplex $\Delta_{15}$.
These are discussed in Examples~\ref{exf32} and~\ref{exex52}.
Explicit matrices that lie on these components are featured in~(\ref{eqnice44matrix}) and in
Examples~\ref{exUabmatrix},~\ref{exgreencurve} and~\ref{example01matrix}.
In Proposition~\ref{prop633}, we resolve a
problem left open in \cite{HRS,HS} concerning the
ML degree arising from $ \partial\mathcal{M}$.
The main result in Section~\ref{sec5} is Theorem~\ref{thmalgebraicboundary}
which characterizes the algebraic boundary of
$m \times n$-matrices of nonnegative rank $3$.
The commutative algebra of the irreducible components in that
boundary is the content of Theorem~\ref{conjmingens}.
Corollary~\ref{cortopoboundary} furnishes a quantifier-free
semialgebraic formula for $\partial\mathcal{M}$.

The proofs of all lemmas, propositions and corollaries
appear in Appendix~\ref{appproofs}.
A~review of basic concepts
in algebraic geometry is given in Appendix~\ref{appbasics}.
This will help the reader understand the technicalities of our main results.
Supplementary materials and software are
posted at the website
\url{http://math.berkeley.edu/\textasciitilde bernd/EM/boundaries.html}.
%
Our readers will find code in {\tt R}, {\tt Macaulay2} and {\tt Magma}
for various sampling experiments,
prime decompositions, semialgebraic formulas and likelihood equations
discussed in this paper.

The methods presented here are not limited to the matrix model
$\mathcal{M}$,
but are applicable to a wide range of statistical models for discrete
data, especially
those used in computational biology \cite{ASCB}.
Such models include phylogenetic models \cite{AMR,ART}
and hidden Markov models \cite{Cri}.
The most immediate generalization is to the $r$th mixture
model of several random variables. It consists of all distributions
corresponding
to tensors of nonnegative rank at most $r$. In other words, we replace
$m \times n$-matrices
by tensors of arbitrary format. The geometry of the case $r=2$ was
studied in
depth by Allman et al. \cite{ARSZ}.
For each of these models, there is a natural EM algorithm, with
an enormous number of stationary points.
The model itself is a complicated semialgebraic set,
and the MLE typically occurs on the boundary of that set. For binary tree
models, this was shown in \cite{ZS}, Section~3.

This article introduces tools needed to gain
a complete understanding of these EM fixed points and
model boundaries. We here study them for the graphical
model in Figure~\ref{figuregraphicalmodel}.
Already in this very simple case,
we discovered patterns that are surprisingly rich.
Thus, the present work serves as a blueprint
for future research in real algebraic geometry that underlies
statistical inference.

\section{Model geometry} \label{sec1}

We begin with a geometric introduction
of the likelihood inference problem to be studied.
Let $\Delta_{mn-1}$ denote the probability simplex of
nonnegative $m \times n$-matrices $P = [p_{ij}]$
with $p_{++} = 1$. Our
model $\mathcal{M}$
is the subset of $\Delta_{mn-1}$
consisting of all matrices of the form
%
\begin{equation}
\label{eqmixformula} P   =   A \cdot\Lambda\cdot B,
\end{equation}
where $A$ is a nonnegative $m \times r$-matrix whose
columns sum to $1$,
$\Lambda$ is a nonnegative $r \times r$ diagonal matrix
whose entries sum to $1$,
and $B$ is a nonnegative $r \times n$-matrix whose
rows sum to~$1$. The triple of parameters $(A,\Lambda,B)$
represents conditional probabilities for the graphical model in
Figure~\ref{figuregraphicalmodel}. In particular, the $k$th column
of $A$ is the conditional probability distribution of $X$ given that $Z
= k$, the $k$th row of $B$ is the conditional probability distribution
given that $Z = k$, and the diagonal of $\Lambda$ is the probability
distribution of $Z$. The parameter space in which $A, \Lambda, B$ lie
is the convex polytope
$ \Theta= (\Delta_{m-1})^r \times\Delta_{r-1} \times(\Delta_{n-1})^r$.
Our model $\mathcal{M}$ is the image of the trilinear map
%
\begin{equation}
\label{eqmapphi} \phi \dvtx    \Theta \rightarrow  \Delta_{mn-1} ,
\qquad
(A, \Lambda, B)    \mapsto     P.
\end{equation}
We seek to learn the model parameters
$(A,\Lambda,B)$ by maximizing the likelihood function
%
\begin{equation}
\label{eqlikelihood} \pmatrix{u_{++}
\cr
u} \cdot\prod
_{i=1}^m \prod_{j=1}^n
p_{ij}^{u_{ij}}
\end{equation}
over $\mathcal{M}$. This is equivalent to maximizing the
log-likelihood function
%
\begin{equation}
\label{log-likelihoodfunction} \ell_U    =    \sum_{i=1}^m
\sum_{j=1}^n u_{ij} \cdot
\operatorname{log} \Biggl(  \sum_{k=1}^r
a_{i k}\lambda_k b_{k j}  \Biggr)
\end{equation}
over $\mathcal{M}$. One
issue that comes up immediately
is that the model parameters are not identifiable:
%
\begin{equation}
\label{eqdimdim}
\qquad\operatorname{dim}(\Theta)  =   r(m+n)-r-1\quad\mbox{but}\quad
\operatorname{dim}(
\mathcal{M})  =  r(m+n)-r^2-1.
\end{equation}
The first expression is the sum of the dimensions of the simplices in
the product that defines
the parameter space $\Theta$.
The second one counts the degrees of freedom in a rank $r$ matrix of
format $m \times n$.
The typical fiber, that is, the preimage of a point in the image of
(\ref{eqmapphi}), is a semialgebraic set of
dimension $r^2-r$. This is the \textit{space of explanations}
whose topology was studied by Mond et al. in \cite{MSS}.
Likelihood inference cannot distinguish among
points in each fiber, so it is preferable to regard MLE
not as an unconstrained optimization problem in $\Theta$
but as a constrained optimization problem in $\mathcal{M}$.
The aim of this paper is to determine its constraints.

Let $\mathcal{V}$ denote the set of
real $m \times n$-matrices $P$ of rank $\leq r$ satisfying $p_{++} =
1$. This set is a variety
because it is given by the vanishing of a set of polynomials, namely,
the $(r+1)\times(r+1)$ minors of the matrix $P$ plus the linear
constraint $p_{++} = 1$.
A point $P \in\mathcal{M}$ is an \textit{interior point} of $\mathcal{M}$
if there is an open ball $U \subset\Delta_{mn-1}$ that contains
$P$ and satisfies $U \cap\mathcal{V} = U \cap\mathcal{M}$.
We call $P \in\mathcal{M}$ a \textit{boundary point} of $\mathcal{M}$
if it is not
an interior point. The set of all such points is
denoted by $\partial\mathcal{M}$ and called the \textit{topological
boundary} of $\mathcal{M}$.
In other words, $\partial\mathcal{M}$ is the boundary of $\mathcal
{M}$ inside~$\mathcal{V}$.
The variety $\mathcal V$ is the Zariski closure of the set $\mathcal
M$; see Appendix~\ref{appbasics}. In other words, the set of
polynomials that vanish on $\mathcal M$ is exactly the same as the set
of polynomials that vanish on $\mathcal V$. Our model $\mathcal M$ is a
full-dimensional subset of
the variety $\mathcal V$ and is given by a set of polynomial
inequalities inside $\mathcal V$.

Fix $U$, $r$ and $P \in\mathcal{M}$ as above. A matrix $P $ is a
nonsingular point on $ \mathcal{V}$ if and only if the
rank of $P$ is exactly $r$. In this case,
its tangent space $\mathrm{T}_P(\mathcal{V})$ has
dimension $r(m+n)-r^2-1$, which, as expected, equals $\dim(\mathcal
M)$. We call $P$
a \textit{critical point} of
the log-likelihood function $\ell_U$ if
$P \in\mathcal{M}$, $P$ is a nonsingular point for~$\mathcal V$, that is,
$\operatorname{rank} (P) = r$,
and the gradient of $\ell_U$ is orthogonal to the tangent space
$\mathrm{T}_P(\mathcal{V})$.
Thus, the critical points are the
nonnegative real solutions of the various \textit{likelihood equations}
derived in \cite{LiAS,HRS,ASCB,ZJG}
to address the MLE problem for~$\mathcal{M}$.
In other words, the critical points are the solutions obtained by using
the Lagrange multipliers method for maximizing the likelihood function
over the set $\mathcal V$.
In the language of algebraic statistics, the critical points are those
points in
$\mathcal{M}$ that are accounted for by the
\textit{ML degree} of the variety~$\mathcal{V}$.

Table~\ref{tablenotcritical} shows
that the global maximum $\widehat P$
of $\ell_U$ is often a noncritical point. This means that the MLE lies
on the topological boundary $\partial\mathcal M$. The ML degree of the
variety $\mathcal{V}$ is irrelevant for
assessing the algebraic complexity of such
$\widehat P$. Instead, we need the ML degree of the boundary,
as given in Proposition~\ref{prop633}, as well as the ML degrees for
the lower-dimensional boundary strata.

The following example illustrates the concepts we have introduced so far
and what they mean.

\begin{example} \label{exUabmatrix}
Fix $m=n = 4$ and $r= 3$.
For any integers $a \geq b \geq0$,
consider the data matrix
%
\begin{equation}
\label{eqUab} U_{a,b} = \lleft[\matrix{ a & a & b & b
\cr
a & b & a & b
\cr
b & a & b & a
\cr
b & b & a & a } \rright].
\end{equation}
Note that $\operatorname{rank}(U_{a,b} ) \leq3$.
For $a=1$ and $ b=0$, this
is the standard example \cite{CR} of a nonnegative
matrix whose nonnegative rank exceeds its rank.
Thus,\vspace*{1pt} $ \frac{1}{8}U_{1,0} $ is a probability distribution
in $\mathcal{V} \setminus\mathcal{M}$.
Within the $2$-parameter family (\ref{eqUab}),
the topological boundary $\partial\mathcal{M}$
is given by the linear equation $ b = (\sqrt{2} - 1) a$.
This follows from the computations in \cite{BCR}, Section~5, and \cite
{MSS}, Section~5.
We conclude that
%
\begin{equation}
\label{eqsqrt2}
\qquad\frac{1}{8(a+b)}U_{a,b}\qquad\mbox{lies in }
\mathcal{V} \setminus\mathcal{M}\quad\mbox{if and only if}\quad b
  <   (\sqrt{2} - 1) a.
\end{equation}

For integers $a > b \geq0$ satisfying (\ref{eqsqrt2}),
the likelihood function (\ref{eqlikelihood}) for $U_{a,b}$ has
precisely eight global maxima on our model $\mathcal{M}$.
These are the following matrices, each divided by $8(a+b)$:
\begin{eqnarray*}
&\displaystyle \lleft[\matrix{ a   &    a   &    b   &    b
\cr
v   &    w   &    t   &    u
\cr
w   &    v   &    u   &    t
\cr
s   &    s   &    r   &    r } \rright]   ,\qquad   \lleft[\matrix
{ v   &
   t   &    w   &    u
\cr
a   &    b   &    a   &    b
\cr
s   &    r   &    s   &    r
\cr
w   &    u   &    v   &    t } \rright]   ,\qquad  \lleft[\matrix{
t   &
   v   &    u   &    w
\cr
r   &    s   &    r   &    s
\cr
b   &    a   &    b   &    a
\cr
u   &    w   &    t   &    v } \rright]   ,&
\\
&\displaystyle \lleft[\matrix
{ r   &
   r   &    s   &    s
\cr
t   &    u   &    v   &    w
\cr
u   &    t   &    w   &    v
\cr
b   &    b   &    a   &    a } \rright]   ,\qquad   \lleft[\matrix
{ a   &
   v   &    w   &    s
\cr
a   &    w   &    v   &    s
\cr
b   &    t   &    u   &    r
\cr
b   &    u   &    t   &    r } \rright]   ,\qquad   \lleft[\matrix
{ v   &
   a   &    s   &    w
\cr
t   &    b   &    r   &    u
\cr
w   &    a   &    s   &    v
\cr
u   &    b   &    r   &    t } \rright]   ,&
\\
&\displaystyle \lleft[\matrix{
t   &
   r   &    b   &    u
\cr
v   &    s   &    a   &    w
\cr
u   &    r   &    b   &    t
\cr
w   &    s   &    a   &    v } \rright]   ,\qquad   \lleft[\matrix
{ r   &
   t   &    u   &    b
\cr
r   &    u   &    t   &    b
\cr
s   &    v   &    w   &    a
\cr
s   &    w   &    v   &    a } \rright] . &
\end{eqnarray*}

This claim can be verified by exact symbolic computation, or by
validated numerics
as in the proof of~\cite{HRS}, Theorem 4.4. Here, $t$ is the unique
simple real root of the cubic equation
\begin{eqnarray*}
&&     \bigl(6 a^3 + 16 a^2 b + 14 a
b^2 + 4 b^3\bigr) t^3   - \bigl(20
a^4 + 44 a^3 b + 8 a b^3 + 32 a^2
b^2\bigr) t^2
\\
&&\quad{} +  \bigl(22 a^5 + 43 a^4 b + 30 a^3
b^2 + 7 a^2 b^3\bigr) t   -     \bigl(8
a^6 + 16 a^5 b + 10 a^4 b^2 + 2
a^3 b^3\bigr)
\\
&&\qquad =     0 .
\end{eqnarray*}
To fill in the other entries of these nonnegative rank $3$ matrices, we
use the rational formulas
\begin{eqnarray*}
s  &=&  \frac{(a + b) t - a^2}{a},\qquad
u  = \frac{t b}{a},
\\
w  &=&  - \frac{t (3 a^2 + 5 a b + 2 b^2) t-4 a^3-5 a^2 b-2 a b^2}{2
a^3 + a^2 b},
\\
r &=& \frac{ 2 a^2 + a b-(a + b) t}{a},
\\
v &=& \frac{(3
a^2 + 5 a b + 2 b^2) t^2 - (6 a^3+8 a^2 b+3 a
b^2) t + 6 a^3 b + 2 a^2 b^2
+ 4 a^4}{2 a^3 + a^2 b}.
\end{eqnarray*}
These formulas represent an exact algebraic solution
to the MLE problem in this case.
They describe the multivalued map
$(a,b) \mapsto\widehat P_{a,b}$
from the data to the eight maximum likelihood estimates. This allows us to
understand exactly how these solutions behave as the matrix entries $a$
and $b$ vary.

The key point is that the eight global maxima
lie in the model boundary $\partial\mathcal{M}$.
They are not critical points of $\ell_U$ on the rank $3$ variety
$\mathcal{V}$.
They will not be found by the methods in
\cite{HRS,ASCB,ZJG}. Instead, we used results about the algebraic boundary
in Section~\ref{sec4} to derive the eight solutions.

We note that this example can be seen as an extension of
\cite{HRS}, Theorem 4.4,
which offers a similar parametric analysis for the data set of the
``100 Swiss Francs Problem'' studied in \cite{FHRZ,ZJG}.
\end{example}

We now introduce the concept of algebraic boundary.
Recall that the topological boundary
$\partial\mathcal{M}$ of the model $\mathcal{M}$ is a semialgebraic subset
inside the probability simplex $\Delta_{mn-1}$. Its dimension is
\[
\operatorname{dim}(\partial\mathcal{M})   =    \operatorname{dim}(\mathcal{M})-1
=
rm+rn-r^2-2.
\]

Any quantifier-free semialgebraic description of $\partial\mathcal
{M}$ will be
a complicated Boolean combination of
polynomial equations and polynomial inequalities.
This can be seen for $r = 3$ in Corollary~\ref{cortopoboundary}.

To simplify the situation,
it is advantageous to relax the inequalities and
keep only the equations. This replaces the
topological boundary of $ \mathcal{M}$
by a much simpler object, namely the algebraic boundary of $\mathcal{M}$.
To be precise, we define the \textit{algebraic boundary} to\vspace*{1pt} be the
Zariski closure $\overline{ \partial\mathcal{M}}$
of the topological boundary $ \partial\mathcal{M}$.
Thus, $\overline{\partial\mathcal{M}}$ is a subvariety of
codimension $1$
inside the variety $\mathcal{V} \subset\PP^{mn-1}$.
Theorem~\ref{thmalgebraicboundary} will show us that
$\overline{\partial\mathcal{M}}$ can have many irreducible components.

The following two-dimensional family of matrices illustrates
the results to be achieved in this paper. These enable us to
discriminate between the topological boundary $\partial\mathcal{M}$
and the algebraic boundary $\overline{\partial\mathcal{M}}$,
and to understand how these boundaries sit inside the variety $\mathcal{V}$.

\begin{example} \label{exgreencurve}
Consider the following $2$-parameter family of $4 \times4$-matrices:
\[
P(x,y) = \lleft[\matrix{ 51 & 9 & 64 & 9
\cr
27 & 63 & 8 & 8
\cr
3 & 34 & 40 & 31
\cr
30 & 25 & 80 & 35 } \rright]  +    x \cdot\lleft[\matrix{ 1 & 1 & 3
& 0
\cr
1 & 0 & 1 & 0
\cr
0 & 1 & 0 & 1
\cr
0 & 0 & 1 & 1 } \rright]  +    y \cdot\lleft[\matrix{ 5 & 4 & 1 &1
\cr
5 & 1 & 5 & 1
\cr
1 & 5 &1 & 5
\cr
1 & 1 & 5 & 5 } \rright]. %
\]
This was chosen so that $P(0,0)$ lies in a unique component of the
topological boundary
$\partial\mathcal{M}$. The equation $\operatorname{det}(P(x,y)) = 0$
defines a plane curve $\mathcal{C}$ of degree~$4$.
This is the thin black curve shown in Figure~\ref{figgreencurve}.
In our family,
this quartic curve $\mathcal{C}$ represents the
Zariski closure $\mathcal{V}$ of the model $\mathcal{M}$.

\begin{figure}

\includegraphics{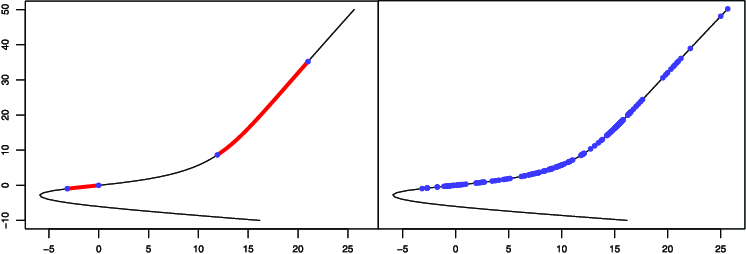}

\caption{In a two-dimensional family of $4 \times4$-matrices,
the matrices of rank $3$ form a quartic curve.
The mixture model, shown in red, has two connected components.
Its topological boundary consists of four points (on the left).
The algebraic
boundary includes many more points (on the right). Currently, there is
no known way to obtain the four points on the topological boundary (in
the left picture) without first considering all points on the algebraic
boundary (in the right picture).}
\label{figgreencurve}
\end{figure}

The algebraic boundary $\overline{\partial\mathcal{M}}$
is the variety described in Example~\ref{exex52}.
The quartic curve $\mathcal{C}$ meets
$\overline{\partial\mathcal{M}}$ in $1618$ real points $(x,y)$.
Of these $1618$ points, precisely $188$ satisfy
the constraint $P(x,y) \geq0$.
These $188$ points are the landmarks for our analysis.
They are shown in blue on the right in Figure~\ref{figgreencurve}.
In addition, we mark
the unique point where the curve $\mathcal{C}$ intersects
the boundary
polygon defined by $P(x,y) \geq0$.
This is the leftmost point, defined by $ \{\operatorname{det}(P(x,y)) = x+5y
+8= 0\}$. It equals
%
\begin{equation}
\label{eqleftendpoint} (-3.161429, -0.967714).
\end{equation}
We examined the $187$ arcs on $\mathcal{C}$ between consecutive points of
$\overline{\partial\mathcal{M}}$ as well as the two arcs at the~ends.
For each arc we checked whether it lies in
$\mathcal{M}$. This was done by a combination of
the EM algorithm in Section~\ref{sec2} and
Theorem~\ref{theoremsemialgebraicdescription}.
Precisely $96$ of the $189$ arcs were found to lie in $\mathcal{M}$.
These form two connected components on the curve $\mathcal{C}$,
namely $19$ arcs between (\ref{eqleftendpoint}) and $(0,0)$,
and
%
\begin{eqnarray}\label{eqnicearc}
&& \mbox{$76$ arcs between}
\nonumber\\[-10pt]\\[-10pt]
&&\qquad (11.905773, 8.642630)
\quad
\mbox{and} \quad(21.001324, 35.202110).\nonumber
\end{eqnarray}
These four points represent the topological boundary $\partial\mathcal{M}$.
We conclude that, in the $2$-dimensional family $P(x,y)$, the model
$\mathcal{M}$
is the union of the two red arcs shown on the left in Figure~\ref
{figgreencurve}.

Our theory of EM fixed points distinguishes between the
(relatively open) red arcs and their blue boundary points.
For the MLE problem, the red points are critical
while the blue points are not critical.
By Table~\ref{tablenotcritical}, the MLE is more likely to be
blue than red, for larger values of $m$ and $n$.
\end{example}

This example demonstrates that the
algebraic methods of Sections~\ref{sec3},~\ref{sec4} and~\ref{sec5} are
indispensable when one desires a reliable analysis of model geometries,
such as that illustrated in Figure~\ref{figgreencurve}.
To apply a method for finding the critical points of a function, for
example, Lagrange multipliers, the domain of the function needs to be
given by equality constraints only. But using only these constraints,
one cannot detect the maxima lying on the topological boundary.
For finding the critical points of the likelihood function on the
topological boundary by using the same methods, one needs to relax the
inequality constraints and consider only the equations defining the
topological boundary. Therefore, one needs to find the critical points
on the algebraic boundary $\overline{\partial\mathcal{M}}$
of the model.

\section{Fixed points of Expectation--Maximization}\label{sec2}

The EM algorithm is an iterative method for finding
local maxima of the likelihood function
(\ref{eqlikelihood}). It can be viewed
as a discrete dynamical system on
the polytope
$ \Theta= (\Delta_{m-1})^r \times\Delta_{r-1} \times(\Delta_{n-1})^r$.
Algorithm \ref{alg1} presents the version in \cite{ASCB}, Section~1.3.

\begin{algorithm}[h!]
\caption{Function EM($U,r$)}\label{alg1}
\begin{algorithmic}
\State Select random $a_1, a_2, \ldots, a_r\in\Delta_{m-1}$,
random $\lambda\in\Delta_{r-1}$, and
random $b_1,b_2,\ldots,b_r\in\Delta_{n-1}$. \\
Run the following steps until the entries of the $m \times n$-matrix
$P$ converge.
\State\textit{\textbf{E-step}: Estimate the $m \times r {\times
} n$-table
that represents this expected hidden data}:

Set $v_{i k j} := \frac{a_{ik}\lambda_kb_{k j}}{\sum_{l=1}^r
a_{il}\lambda_lb_{lj}}u_{i j}$
for $ i=1,\ldots,m$, $k=1,\ldots,r$ and $j=1,\ldots,n$.
\State\textit{\textbf{M-step}: Maximize the likelihood function
of the model}
\includegraphics{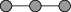} for the hidden data:

Set $\lambda_k := \sum_{i=1}^m \sum_{j=1}^n v_{i k j}/u_{++}$
for $k=1,\ldots,r$.

Set $a_{ik} := (\sum_{j=1}^n v_{i k j})/(u_{++}\lambda_k)$ for
$k=1,\ldots,r$ and $i=1,\ldots,m$.

Set $b_{k j} := (\sum_{i=1}^m v_{i k j})/(u_{++}\lambda_k)$
for $k=1,\ldots,r$ and $j=1,\ldots,n$.

\State\textit{\textbf{Update} the estimate of the joint
distribution for our mixture model} \includegraphics{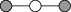}:

Set $p_{i j} := \sum_{k=1}^r a_{ik}\lambda_k b_{k j}$
for $ i=1,\ldots,m$ and $j = 1,\ldots, n$.

\State Return $P$.
\end{algorithmic}
\end{algorithm}

The alternating sequence of E-steps and M-steps defines
trajectories in the parameter polytope $\Theta$.
The log-likelihood function (\ref{log-likelihoodfunction})
is nondecreasing along each trajectory (cf.~\cite{ASCB}, Theorem 1.15).
In fact, the value can stay the same only
at a fixed point of the EM algorithm.
See Dempster et al.~\cite{DLR} for
the general version of EM and
its increasing behavior and convergence.

\begin{definition}
An \textit{EM fixed point} for a given table $U$ is
any point $(A, \Lambda, B)$
in the polytope
$\Theta= (\Delta_{m-1})^r\times\Delta_{r-1}\times(\Delta_{n-1})^r $
to which the EM algorithm can converge if it is
applied to $(U,r)$.
\end{definition}

Every global maximum $\widehat P$ of $\ell_U$ is among the
EM fixed points.
One hopes that $\widehat P$
has a large basin of attraction, and that
the initial parameter choice $(A,\Lambda,B)$ gives a
trajectory that converges to $\widehat P$. However, this need not be
the case,
since the EM dynamics on $\Theta$ has many fixed points
other than $\widehat P$. Our aim is to understand all of these.

\begin{example}\label{example01matrix}
The following data matrix is
obtained by setting $a=1,b=0$ in Example~\ref{exUabmatrix}:
\[
U  =  \lleft[\matrix{ 1 & 1 & 0 & 0
\cr
1 & 0 & 1 & 0
\cr
0 & 1 & 0 & 1
\cr
0 & 0 & 1 & 1 } \rright].
\]
Among the EM fixed points for this choice of $U$ with $r=3$ we find
the probability distributions
\begin{eqnarray*}
 P_1  &=&   \frac{1}{24} \lleft[
\matrix{ 3 & 3 & 0 & 0
\cr
2 & 0 & 4 & 0
\cr
0 & 2 & 0 & 4
\cr
1 & 1 & 2 & 2 } \rright] , \qquad P_2   =   \frac{1}{16}
\lleft[\matrix{ 2 & 2 & 0 & 0
\cr
2 & 0 & 2 & 0
\cr
0 & 1 & 1 & 2
\cr
0 & 1 & 1 & 2 } \rright]\quad\mbox{and}
\\[3pt]
P_3  &=&
\frac{1}{48} \lleft[\matrix{ 4 & 8 & 0 & 0
\cr
3 & 0 & 4 & 5
\cr
5 & 4 & 0 & 3
\cr
0 & 0 & 8 & 4 } \rright],
\end{eqnarray*}
and their orbits under the symmetry group of $U$.
For instance, the orbit of $P_1$ is obtained by setting
$s = \frac{1}{3}, r = \frac{2}{3}, v = \frac{2}{3}, t = \frac
{4}{3}, w = u = 0$
in the eight matrices in Example~\ref{exUabmatrix}.
Over $98\%$ of our runs with random starting points in $\Theta$
converged to one of these eight global maximizers of $\ell_U$.
Matrices in the orbits of $P_2$, respectively, $P_3$
were approached only rarely (less than $2 \%$) by the EM algorithm.
\end{example}

\begin{lemma}\label{lemmaEMfixed}
The following are equivalent
for a point $(A,\Lambda,B)$ in the parameter polytope~$\Theta$:
\begin{longlist}[(3)]
\item[(1)] The point $(A, \Lambda, B)$ is an EM fixed point.
\item[(2)] If we start EM with $(A, \Lambda, B)$
instead of a random point, then
EM converges to $(A, \Lambda, B)$.
\item[(3)] The point $(A, \Lambda, B)$ remains fixed after one
completion of the E-step and the M-step.
\end{longlist}
\end{lemma}



It is often believed (and actually stated in \cite{ASCB}, Theorem 1.5)
that every EM fixed point is a critical point of the log-likelihood
function $\ell_U$.
This statement is not true for the definition of ``critical''
given in Section~\ref{sec1}. In fact, for many instances~$U$,
the global maximum $\widehat P$ is not critical.

To underscore this important point and its statistical relevance,
we tested the EM algorithm
on random data matrices $U$
for a range of models with $m=n$.
The following example explains
Table~\ref{tablenotcritical}.

\begin{example}
\label{remtable1experiment}
In our first simulation, we generated random matrices $U$
from the uniform distribution on $\Delta_{mn-1}$ by using \texttt{R}
and then scaling to get integer entries. For each matrix $U$, we
ran the EM algorithm $2000$ times
to ensure convergence with high probability to the global maximum
$\widehat P$
on $\mathcal{M}$. Each run had $2000$ steps. We then checked whether
$\widehat P$ is a critical point of $\ell_U$
using the rank criterion in \cite{HRS}, equation~(2.3).
Our results are reported in Table~\ref{tablenotcritical}.
The main finding is that,
with high probability as the matrix size increases, the MLE $\widehat P$
lands on the topological boundary $\partial\mathcal{M}$, and it
fails to be critical.

In a second simulation, we started with matrices $A\in\mathbb
N^{m\times r}$ and
$B\in\mathbb N^{r\times n}$ whose entries were sampled uniformly from
$\{0,1,\ldots,100\}$.
We then fixed $P \in\mathcal{M}$ to be the $m\times n$ probability
matrix given by
$AB$ divided by the sum of its entries.
We finally took $T m n$ samples from the distribution $P$ and recorded
the results
in an $m \times n$ data matrix $U$. Thereafter, we applied EM to $U$.
We observed the following.
If $T \geq20$ then the fraction of times the MLE lies in $\partial
\mathcal M$
is very close to $0$. When $T \leq10$ though, this fraction was higher
than the
results reported in Table~\ref{tablenotcritical}. For $T = 10$ and $m=n=4$, $r=3$, this
fraction was
$13$\%, for $m=n=5, r=3$, it was $23$\%, and for $m=n=5, r=4$, it was
$17$\%.
Therefore, based on these experiments,
in order to have the MLE be a critical point in $\mathcal M$, one
should have at least 20 times more samples than entries of the matrix.
\end{example}

This brings our attention to the
problem of identifying the fixed points of EM. If we could compute all
EM fixed points,
then this would reveal
the global maximizer of $\ell_U$. Since a point is EM fixed if and
only if it stays fixed after an E-step and an M-step,
we can write rational function equations for the EM fixed points
in~$\Theta$:
\begin{eqnarray*}
\lambda_k  &=&  \frac{1}{u_{++}} \sum
_{i=1}^m\sum_{j=1}^n
\frac
{a_{ik}\lambda_k b_{kj}}{\sum_{l=1}^ra_{il}\lambda_l b_{lj}}u_{ij}
\qquad\mbox{for all } k,
\\
a_{ik}  &=& \frac{1}{\lambda_k u_{++}}\sum_{j=1}^n
\frac
{a_{ik}\lambda_k b_{kj}}{\sum_{l=1}^ra_{il}\lambda_l b_{lj}}u_{ij}
\qquad\mbox{for all } i,k,
\\
b_{kj}  &=&  \frac{1}{\lambda_k u_{++}}\sum_{i=1}^m
\frac{a_{ik}\lambda_k b_{kj}}{\sum_{l=1}^ra_{il}\lambda_l b_{lj}}u_{ij}
\qquad\mbox{for all }k, j .
\end{eqnarray*}
Our goal is to understand the solutions to these equations for a fixed
positive matrix~$U$.
We seek to find
the variety they define in the polytope $\Theta$
and
the image of that variety in~$\mathcal{M} $.

In the EM algorithm, we usually start with parameters $a_{ik}, \lambda
_k, b_{kj}$
that are strictly positive. The $a_{ik}$ or $b_{kl}$ may become
zero in the limit, but the parameters $\lambda_k$ always
remain positive when the $u_{ij}$ are positive since the entries of
each column of $A$ and each row of $B$ sum to $1$. This justifies that
we cancel out the factors $\lambda_k$ in our equations.
After this, the first equation is implied by the other two.
Therefore, the set of all EM fixed points is a variety, and it is
characterized by
\begin{eqnarray*}
a_{ik}  &=&  \frac{1}{u_{++}}\sum_{j=1}^n
\frac{a_{ik}b_{kj}}{\sum_{l=1}^ra_{il}\lambda_l b_{lj}}u_{ij} \qquad
\mbox{for all } i,k ,
\\
b_{kj}  &=& \frac{1}{u_{++}}\sum_{i=1}^m
\frac{a_{ik}b_{kj}}{\sum_{l=1}^ra_{il}\lambda_l b_{lj}}u_{ij} \qquad
\mbox{for all }k, j.
\end{eqnarray*}
Suppose that a denominator $\sum_l a_{il}\lambda_lb_{l j} $ is zero
at a point in $\Theta$. Then
$a_{ik}b_{kj} = 0$ for all $k$, and
the expression $\frac{a_{ik}b_{kj}}{\sum_{l=1}^ra_{il}\lambda_l b_{lj}}$
would be considered $0$. Using the identity
$  p_{ij}  =   \sum_{l=1}^ra_{il}\lambda_l b_{lj}$,
we can rewrite our two fixed point equations in the form
%
\begin{eqnarray}\label{eqwithdenom}
a_{ik} \Biggl(\sum_{j=1}^n \biggl( u_{++} - \frac{u_{ij}}{p_{i j}} \biggr) b_{kj} \Biggr) &=&
0\qquad\mbox{for all }k, i  \quad\mbox{and}
\nonumber\\[-8pt]\\[-8pt]\nonumber
b_{kj} \Biggl(\sum_{i=1}^m \biggl( u_{++} -
\frac{u_{ij}}{p_{i j}} \biggr) a_{ik} \Biggr) &=& 0\qquad \mbox{for all }k, j.
\end{eqnarray}
Let\vspace*{1pt} $R$ denote the $m\times n$ matrix with entries $r_{i j}= u_{++} -
\frac{u_{ij}}{p_{i j}} $.
The matrix $R$ is the gradient of the log-likelihood function
$ \ell_U(P)$, as seen in~\cite{HRS}, equation~(3.1). With this, our
fixed point equations are
%
\begin{eqnarray}\label{EMfixed}
a_{ik} \Biggl(\sum_{j=1}^n
r_{i j}b_{kj} \Biggr)  &=&  0    \qquad\mbox{for all }k, i
\quad\mbox{and}
\nonumber\\[-8pt]\\[-8pt]\nonumber
b_{kj} \Biggl(\sum
_{i=1}^m r_{i j}a_{ik} \Biggr)
&=&  0     \qquad\mbox{for all }k, j.
\end{eqnarray}
We summarize our discussion in the following theorem,
with (\ref{EMfixed}) rewritten in matrix form.

\begin{theorem} \label{thmEMfixedTheta}
The variety of EM fixed points in the polytope $\Theta$ is defined by
the equations
%
\begin{equation}
\label{EMfixed2} A\star\bigl(R\cdot B^T\bigr) = 0,\qquad B\star
\bigl(A^T\cdot R\bigr) = 0,
\end{equation}
where $R$ is the gradient matrix of the log-likelihood function
and $\star$ denotes the Hadamard product.
The subset of EM fixed points that are critical points is defined by
$R \cdot B^T = 0$ and $A^T \cdot R = 0$.
\end{theorem}

\begin{pf}
Since (\ref{EMfixed2}) is equivalent to (\ref{EMfixed}),
the first sentence is proved by the derivation above.
For the second sentence, we
consider the normal space
of the variety $\mathcal{V}$
at a rank $r$ matrix $P = A \Lambda B$.
This is the orthogonal complement of the~tangent space
$\mathrm{T}_P(\mathcal{V})$. The normal space can be expressed as
the kernel of the linear map $Q \mapsto(Q \cdot B^T, A^T \cdot Q)$.
Hence, $R = \operatorname{grad}_P (\ell_U)$ is perpendicular to
$\mathrm{T}_P(\mathcal{V})$ if and only if
$R \cdot B^T=0$ and $A^T \cdot R= 0$.
Therefore, the polynomial equations
(\ref{EMfixed2}) define the Zariski closure of the set of
parameters for which $P$ is critical.
\end{pf}

The variety defined by (\ref{EMfixed2}) is reducible.
In Section~\ref{sec4},
we shall present a detailed study of its irreducible components,
along with a discussion of their statistical interpretation.
As a preview, we here decompose the variety of
EM fixed points in the simplest possible case.

\begin{example}
\label{exEM22}
Let $m = n = 2$, $r = 1$, and consider the ideal generated
by the cubics in (\ref{EMfixed2}):
\begin{eqnarray*}
\mathcal{F}    &=&     \bigl\langle a_{11} (r_{11}
b_{11} + r_{12} b_{12}),  a_{21}
(r_{21} b_{11} + r_{22} b_{12}),
\\
&&\hspace*{6pt}{} b_{11} (a_{11} r_{11} + a_{21}
r_{21}),  b_{12} (a_{11} r_{12} +
a_{21} r_{22}) \bigr\rangle. %
\end{eqnarray*}
The software \texttt{Macaulay2} \cite{M2} computes a primary decomposition
into $12$ components:
%
\begin{eqnarray}
\label{primarydecompositionmix22} %
\mathcal{F}  &=  & \langle r_{11}r_{22}
- r_{12}r_{21}, a_{11} r_{11}+a_{21}r_{21},
a_{11}r_{12}+a_{21}r_{22},b_{11}r_{11}\nonumber
\\
&&\hspace*{135pt}
{}+b_{12}r_{12},
b_{11}r_{21}+b_{12}r_{22}  \rangle\nonumber
\\
&&{} \cap \langle a_{11} , r_{21},r_{22}\rangle \cap
  \langle a_{21} ,r_{11},r_{12} \rangle \cap
\langle r_{12},r_{22},b_{11}\rangle \cap \langle
r_{11},r_{21},b_{12} \rangle
\\
&&{} \cap \langle a_{11},r_{22},b_{11}\rangle \cap
\langle a_{11}, r_{21},b_{12} \rangle \cap
\langle a_{21},r_{12},b_{11} \rangle \cap \langle
a_{21},r_{11},b_{12}\rangle\nonumber
\\
&&{} \cap  \langle a_{11},a_{21}\rangle \cap  \langle
b_{11},b_{12}\rangle   \cap   \bigl( \langle
a_{11},a_{21} \rangle^2 + \langle
b_{11}, b_{12} \rangle^2 + \mathcal{F} \bigr).\nonumber
\end{eqnarray}
The last primary ideal is embedded.
Thus, $\mathcal{F}$ is not a radical ideal. Its radical
requires an extra generator of degree $5$.
The first $11$ ideals in (\ref{primarydecompositionmix22}) are the
minimal primes
of~$\mathcal{F}$. These give the irreducible components of the variety
$V(\mathcal{F})$.
The first ideal represents the critical points in $\mathcal{M}$.
\end{example}

\section{Matrices of nonnegative rank three}\label{sec3}

While the EM algorithm operates in the polytope $\Theta$
of model parameters $(A,\Lambda,B)$, the mixture model $\mathcal{M}$ lives
in the simplex $\Delta_{mn-1} \subset\mathbb{R}^{m \times n}$
of all joint distributions. The parame\-trization $\phi$ is not identifiable.
The topology of its fibers was studied by Mond et al.~\cite{MSS},
with focus on the first nontrivial case,
when the rank $r$ is three.
We build on their work to derive a semialgebraic
characterization of $\mathcal{M}$. This section is self-contained. It
can be
read independently from our earlier discussion of the EM algorithm. It
is aimed at
all readers interested in nonnegative matrix factorization,
regardless of its statistical relevance.

We now fix $r = 3$. Let
$A$ be a real $m \times3$-matrix with rows
$a_1,\ldots,a_m$, and
$B$ a~real $3 \times n$-matrix with columns
$b_1,\ldots,b_n$.
The vectors $b_j \in\R^3$ represent points
in the projective plane $\mathbb{P}^2$.
We view the $a_i$ as elements
in the dual space $(\R^3)^*$. These represent lines in $\mathbb{P}^2$.
Geometric algebra (a.k.a.~Grassmann--Cayley algebra~\cite{Whi})
furnishes two bilinear operations,
\[
\vee\dvtx  \R^3 \times\R^3 \rightarrow\bigl(\R^3
\bigr)^* \quad\mbox{and} \quad\wedge\dvtx  \bigl(\R^3\bigr)^* \times\bigl(
\R^3\bigr)^* \rightarrow\R^3 . %
\]
These correspond to the classical cross product in $3$-space.
Geometrically, $a_i \wedge a_j$ is the intersection point
of the lines $a_i$ and $a_j$ in $\mathbb{P}^2$, and
$b_i \vee b_j$ is the line spanned by the points $b_i$ and~$b_j$
in $\mathbb{P}^2$.
The pairing $(\R^3)^* \times\R^3 \rightarrow\R$ can be denoted
by either $\vee$ or $\wedge$. With these conventions, the operations
$\vee$ and $\wedge$ are alternating, associative and distributive.
For instance, the minor
%
\begin{equation}
\label{eqdeterminant} a_i \wedge a_j \wedge a_k
  =   \operatorname{det}(a_i,a_j,a_k)
\end{equation}
vanishes if and only if the lines $a_i, a_j$ and $a_k$ are concurrent.
Likewise, the polynomial
%
\begin{eqnarray}
\label{eqtwotwo}
&& (a_i \wedge a_j) \vee
b_{i'} \vee b_{k'}\nonumber
\\
&&\qquad  =  a_{i1} a_{j2}
b_{1i'} b_{2k'}-a_{i1}a_{j2}
b_{1k'} b_{2i'}+a_{i1} a_{j3}
b_{1i'} b_{3k'} -a_{i1} a_{j3}
b_{1k'} b_{3i'}
\nonumber\\[-8pt]\\[-8pt]\nonumber
&&\quad\qquad{} -a_{i2} a_{j1} b_{1i'} b_{2k'} +
a_{i2} a_{j1} b_{1k'} b_{2i'} +
a_{i2} a_{j3} b_{2i'} b_{3k'} -
a_{i2} a_{j3} b_{2k'} b_{3i'}\nonumber
\\
&&\quad\qquad{} - a _{i3} a_{j1} b_{1i'} b_{3k'} +
a_{i3} a_{j1} b_{1k'} b_{3i'} -
a_{i3} a_{j2} b_{2i'} b_{3k'} +
a_{i3} a_{j2} b_{2k'} b_{3i'} \nonumber
\end{eqnarray}
expresses the condition that the
lines $a_i$ and $a_j$ intersect in a point
on the line given by
$b_{i'}$ and~$b_{k'}$.
Of special interest is the following
formula involving four rows of $A$ and three columns of~$B$:
%
\begin{equation}
\label{eqsixthree} \bigl( \bigl((a_i \wedge a_j)\vee
b_{i'} \bigr) \wedge a_k \bigr) \vee\bigl( \bigl(
(a_i \wedge a_j) \vee b_{j'}\bigr) \wedge
a_l \bigr) \vee b_{k'}.
\end{equation}
Its expansion is a bihomogeneous polynomial
of degree $(6,3)$ with $330$ terms in~$(A,B)$.

A matrix $P \in\R^{m \times n}$ has
nonnegative rank $\leq3$ if it admits a
factorization $P = A B$ with $A$ and $B$
nonnegative. The set of such matrices $P$
with $p_{++} = 1$ is precisely the mixture model $\mathcal{M}$
discussed in the earlier sections.
Comparing with (\ref{eqmixformula}), we here subsume
the diagonal matrix $\Lambda$ into either $A$ or $B$.
In what follows, we consider the set
$\mathcal{N}$ of\vspace*{1pt} pairs $(A,B)$ whose product
$AB$ has nonnegative rank $\leq3$.
Thus, $\mathcal{N}$ is a semialgebraic subset of
$\R^{m \times3} \oplus\R^{3 \times n}$. We shall prove:

\begin{theorem}\label{theoremsemialgebraicdescription}
A pair $(A,B)$ is in
$\mathcal{N}$ if and only if
$AB \geq0$ and the following condition~holds: either
$\rank(AB)<3$, or $\rank(AB)=3$ and there exist indices
$i,j \in[m]$, $i',j' \in[n]$
such~that:
\begin{longlist}[i]
\item[] $\operatorname{sign}(\ref{eqdeterminant})$ is the same
or zero for all $k \in[m] \setminus\{i,j\}$,

\item[] and   $\operatorname{sign}(\ref{eqtwotwo})$ is the same or zero for
all $k'\in[n]\setminus\{i'\}$,

\item[] and   $\operatorname{sign}((\ref{eqtwotwo})[i'   \rightarrow  j'])$ is the same
or zero for
all $k'\in[n]\setminus\{j'\}$,

\item[] and   {(\ref{eqsixthree})} $\cdot$
{(\ref{eqsixthree}) $[k   \leftrightarrow  l]$} $ \geq 0  $
for all
$  \{k,l\}\subseteq[m]\setminus\{i,j\}$ and $k'\in
[n]\setminus\{i',j'\} $,

\item[]  or there exist $i,j \in[n]$, $i',j' \in[m]$
such that these conditions hold after swapping $A$ with $B^T$.
\end{longlist}
\end{theorem}

Here, $[m] = \{1,2,\ldots,m\}$, and
the notation $[i'   \rightarrow  j']$ means that the index $i'$ is
replaced by the index $j'$
in the preceding expression, and $[k   \leftrightarrow  l]$ means
that $k$ and $l$ are switched.

Theorem~\ref{theoremsemialgebraicdescription}
is our main result in Section~\ref{sec3}. It gives a finite disjunction
of conjunctions of polynomial inequalities in $A$ and $B$,
and thus a quantifier-free first order formula for $\mathcal{N}$.
This represents our mixture model as follows:
to test whether $P $ lies in $\mathcal{M}$, check whether $\operatorname{rank}(P) \leq3$;
if yes, compute any rank $3$ factorization $P = AB$ and
check whether $(A,B)$ lies in $\mathcal{N}$.
Code for performing these computations in {\tt Macaulay2}
is posted on our website.

\begin{figure}

\includegraphics{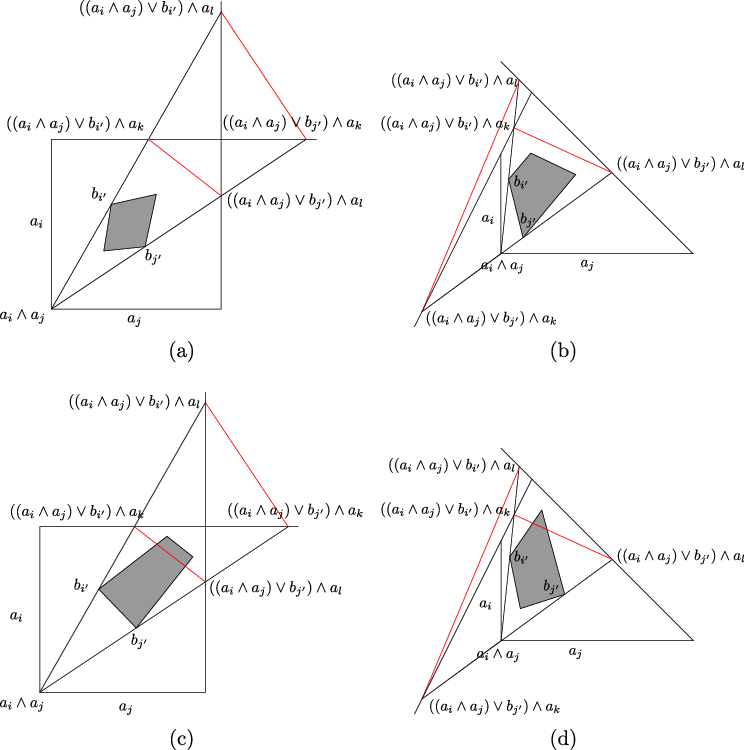}

\caption{In the diagrams \textup{(a)} and \textup{(b)}, the conditions of
Theorem~\protect\ref{theoremsemialgebraicdescription} are satisfied
for the chosen $i,j,i',j'$. In the
diagrams \textup{(c)} and \textup{(d)}, the conditions of Theorem~\protect\ref
{theoremsemialgebraicdescription} fail for the chosen $i,j,i',j'$.}\label{figsemialgebraicdescription}
\end{figure}

Theorem~\ref{theoremsemialgebraicdescription} is an algebraic
translation of a geometric algorithm. For an illustration, see
Figure~\ref{figsemialgebraicdescription}.
In the rest of the section, we will study the geometric description of
nonnegative rank that leads to the algorithm.
Let $P$ be a nonnegative $m \times n$ matrix of rank $r$.
We write
$\operatorname{span}(P)$ and $\cone(P)$ for the linear space and
the cone spanned by the columns of $P$, and we define
%
\begin{eqnarray}
\mathcal{A}&=&\operatorname{span}(P) \cap\Delta_{m-1} \quad\mbox{and}\quad
\mathcal{B}=\cone(P) \cap\Delta_{m-1}.
\end{eqnarray}

The matrix $P$ has a size $r$ nonnegative factorization if and only if
there exists a polytope $\Delta$ with $r$ vertices such that $\mathcal
{B} \subseteq\Delta\subseteq\mathcal{A}$; see
\cite{MSS}, Lemma~2.2. Without loss of generality, we will assume in
the rest of this section that the vertices of $\Delta$ lie on the
boundary of $\mathcal{A}$. We write $\mathcal{M}_r$ for the
set of $m \times n$-matrices of nonnegative rank $\leq r$.
Here is an illustration that is simpler than Example~\ref{exgreencurve}:

\begin{example}
In \cite{FH}, Section~2.7.2, the following family of matrices of rank${}\leq3$ is considered:
%
\begin{eqnarray}
P(a,b)= \lleft[\matrix{ 1-a & 1+a & 1+a & 1-a
\cr
1-b & 1-b & 1+b & 1+b
\cr
1+a & 1-a & 1-a & 1+a
\cr
1+b & 1+b & 1-b & 1-b } \rright].
\end{eqnarray}
Here, $\mathcal{B}$ is a rectangle and $\mathcal{A}=\{x\in\Delta
_3\dvtx x_1-x_2+x_3-x_4=0\}$ is a square, see Figure~\ref
{fignonnegativerankgeometricdescription}.
Using Theorem~\ref{theoremsemialgebraicdescription},
we can check that $P(a,b)$ lies in $\mathcal{M}_3$
if and only if $ab + a + b \leq1$.
%
\begin{figure}

\includegraphics{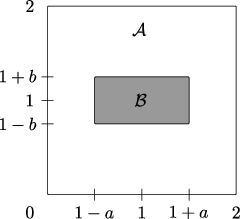}

\caption{The matrix $P(a,b)$ defines a nested pair of rectangles.}\label{fignonnegativerankgeometricdescription}
\end{figure}
\end{example}

\begin{lemma}\label{lemmageometricdescriptionofinteriorpoints}
A matrix $P\in\R^{m\times n}_{\geq0}$ of rank $r$ lies in the
interior of $\mathcal{M}_r$ if and only if there exists an
$(r-1)$-simplex $\Delta\subseteq\mathcal{A}$ such that $\mathcal
{B}$ is contained in the interior of $\Delta$.
It lies on the boundary of $\mathcal{M}_r$ if and only if
every $(r-1)$-simplex $\Delta$ with $\mathcal{B} \subseteq\Delta
\subseteq\mathcal{A}$ contains a vertex of $\mathcal{B}$ on its boundary.
\end{lemma}

\begin{figure}

\includegraphics{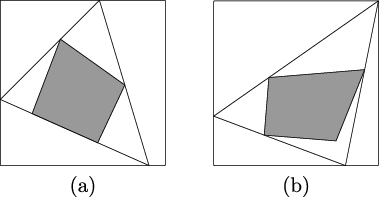}

\caption{Critical configurations.}
\label{figcriticalconfigurations}
\end{figure}

For $r=3$, Mond et al. \cite{MSS} prove the following result.
Suppose $\mathcal{B}\subseteq\Delta\subseteq\mathcal{A}$
and every edge of $\Delta$ contains a vertex of $\mathcal{B}$.
Then $t\mathcal{B} \subseteq\Delta' \subseteq\mathcal{A}$
for some triangle $\Delta'$ and some $t>1$, unless:
\begin{longlist}[(a)]
\item[(a)]
an edge of $\Delta$ contains an edge of $\mathcal{B}$, or
\item[(b)] a vertex of $\Delta$ coincides with a vertex of $\mathcal{A}$.
\end{longlist}
Here, the dilate $t\mathcal{B}$ is taken with respect to a point in
the interior of $\mathcal{B}$.
By Lemma~\ref{lemmageometricdescriptionofinteriorpoints}, this
means that $P$ lies
in the interior of $\mathcal{M}_3^{m\times n}$ unless one of (a) and~(b) holds.
The conditions (a) and (b) are shown in
Figure~\ref{figcriticalconfigurations}.
For the proof of this result, we refer to~\cite{MSS}, Lemmas~3.10~and~4.3.

\begin{corollary}\label{corollaryboundarypointsforrankthreematrices}
A matrix $P\in\mathcal{M}_3$ lies on the boundary of $\mathcal{M}_3$
if and only~if:
\begin{itemize}
\item$P$ has a zero entry, or
\item$\rank(P)=3$ and if $\Delta$ is any triangle with $\mathcal{B}
\subseteq\Delta\subseteq\mathcal{A}$ then every edge of $\Delta$
contains a vertex of $\mathcal{B}$, and \textup{(a)} or \textup{(b)} holds.
\end{itemize}
\end{corollary}

\begin{corollary}\label{corollarygeometricconditionsforamatrixtohavenonnegativerank3}
A matrix $P\in\R^{m\times n}_{\geq0}$
has nonnegative rank $ \leq3$ if and only if:
\begin{itemize}
\item$\rank(P)<3$, or
\item$\rank(P)=3$ and there exists a triangle $\Delta$ with
$\mathcal{B} \subseteq\Delta\subseteq\mathcal{A}$ such that a
vertex of $\Delta$ coincides with a vertex of $\mathcal{A}$, or
\item$\rank(P)=3$ and there exists a triangle $\Delta$ with
$\mathcal{B} \subseteq\Delta\subseteq\mathcal{A}$ such that an
edge of $\Delta$ contains an edge of $\mathcal{B}$.
\end{itemize}
\end{corollary}

Corollary~\ref{corollarygeometricconditionsforamatrixtohavenonnegativerank3}
provides a geometric algorithm similar to that of Aggarwal et~al. \cite
{ABRS} for checking
whether a matrix has
nonnegative rank $3$. For the algorithm, we need to consider one
condition for every vertex of $\mathcal{A}$ and one condition for
every edge of $\mathcal{B}$. We now explain these conditions.

Let $v$ be a vertex of $\mathcal{A}$. Let $b_1,b_2$ be the vertices of
$\mathcal{B}$ such that
$l_1 = \overline{v b_1} $ and $l_2 = \overline{v b_2}$
support $\mathcal{B}$. Let $\Delta$ be the convex hull of $v$ and the
other two intersection points of the lines $l_1,l_2$ with the boundary
of $\mathcal{A}$. If $\mathcal{B} \subseteq\Delta$, then $P$ has
nonnegative rank $3$.

Let $l$ be the line spanned by an edge of $\mathcal{B}$. Let $v_1,v_2$
be the intersection points of $l$ with
$\partial\mathcal{A}$. Let $b_1,b_2$ be the vertices of $\mathcal
{B}$ such that
$l_1 = \overline{v_1 b_1}$ and
$l_2 = \overline{v_2 b_2}$ support~$\mathcal{B}$. Let $v_3$ be the
intersection point of $l_1$ and $l_2$. If $\conv(v_1,v_2,v_3)\subseteq
\mathcal{A}$, then $P$ has nonnegative rank $3$.

\begin{pf*}{Proof of Theorem~\ref{theoremsemialgebraicdescription}}
Let $\rank(P)=3$ and consider
any factorization $P = AB$ where
$a_1,\ldots,a_m\in(\R^3)^*$ are the row vectors of $A$
and $b_1,\ldots,b_n\in\R^3$ are the column vectors of $B$.
The map $x \mapsto Ax$ identifies $\R^3$ with the common column space
of $A$ and $P$. Under this identification, and by passing from
\mbox{$3$-}dimensional cones
to polygons in $\R^2$,
we can assume that the edges of $\mathcal{A}$ are given by $a_1,\ldots
,a_m$ and the vertices of $\mathcal{B}$ are given by $b_1,\ldots,b_n$.

To test whether $P$ belongs to $\mathcal{M}_3$,
we use the geometric conditions
in Corollary~\ref{corollarygeometricconditionsforamatrixtohavenonnegativerank3}.
These still involve a quantifier over $\Delta$. Our aim is to translate
them into the given quantifier-free formula, referring only to the
vertices $b_i$ of $\mathcal{B}$ and the
edges $a_j$ of $\mathcal{A}$.
First, we check with the sign condition on (\ref{eqdeterminant}) that
the intersection point $a_i\wedge a_j$ defines a vertex of $\mathcal
{A}$. Next we verify that the lines $(a_i\wedge a_j)\vee b_{i'}$ and
$(a_i\wedge a_j)\vee b_{j'}$ are supporting $\mathcal{B}$, that
is,~all vertices of $\mathcal{B}$ lie on the same side of the lines
$(a_i\wedge a_j)\vee b_{i'}$ and $ (a_i\wedge a_j)\vee b_{j'}$. For
this, we use the sign conditions on
(\ref{eqtwotwo}) and (\ref{eqtwotwo}) $[i'   \rightarrow  j']$.

Finally, we need to check whether all vertices of $\mathcal{B}$ belong
to the convex hull of $a_i\wedge a_j$ and the other two intersection
points of the lines $(a_i\wedge a_j)\vee b_{i'}$ and $(a_i\wedge
a_j)\vee b_{j'}$ with the boundary of $\mathcal{A}$. Fix $\{k,l\}
\subseteq[m]\setminus\{i,j\}$. If either the line $(a_i\wedge
a_j)\vee b_{i'}$ intersects $a_k$ or the line $(a_i\wedge a_j)\vee
b_{j'}$ intersects $a_l$ outside $\mathcal{A}$, then the polygon
$\mathcal{B}$ lies completely on one side of the line $(((a_i\wedge
a_j)\vee b_{i'}) \wedge a_k) \vee(((a_i\wedge a_j)\vee b_{j'})\wedge
a_l)$. Similarly, if either the line $(a_i\wedge a_j)\vee b_{i'}$
intersects $a_l$ or the line $(a_i\wedge a_j)\vee b_{j'}$ intersects
$a_k$ outside $\mathcal{A}$, then the polygon $\mathcal{B}$ lies
completely on one side of the line $(((a_i\wedge a_j)\vee b_{i'})
\wedge a_l) \vee(((a_i\wedge a_j)\vee b_{j'})\wedge a_k)$. Then the
condition (\ref{eqsixthree}) $\cdot$ (\ref{eqsixthree}) $[k
\leftrightarrow  l]\geq0$ is automatically satisfied
for all $ k'\in[n]\setminus\{i',j'\}$. If the intersection points
$((a_i\wedge a_j)\vee b_{i'})\wedge a_k$ and $((a_i\wedge a_j)\vee
b_{j'})\wedge a_l$ are on the boundary of $\mathcal{A}$, then the
polygon $\mathcal{B}$ is on one side of $(((a_i\wedge a_j)\vee
b_{i'})\wedge a_l)\vee(((a_i\wedge a_j)\vee b_{j'})\wedge a_k)$.
In this case, we use the conditions
(\ref{eqsixthree}) $\cdot$ (\ref{eqsixthree}) $[k
\leftrightarrow  l]\geq0$
to check whether $\mathcal{B}$ is also on one side of
the line $(((a_i\wedge a_j)\vee b_{i'})\wedge a_k)\vee(((a_i\wedge
a_j)\vee b_{j'})\wedge a_l)$.
For an illustration, see Figure~\ref{figsemialgebraicdescription}.
\end{pf*}

We wish to reiterate that the semialgebraic formula for our model
in Theorem~\ref{theoremsemialgebraicdescription}
is quantifier-free. It is a finite Boolean combination of polynomial
inequalities with rational coefficients.

\begin{corollary}\label{corollaryrationalfactorizations}
If a rational $m \times n$ matrix $P $ has nonnegative rank $\leq3$, then
there exists a nonnegative rank $\leq3$ factorization
$P = AB$ where all entries of $A$ and $B$ are rational numbers.
\end{corollary}


This answers a question of Cohen and Rothblum in~\cite{CR} for
matrices of nonnegative rank $3$. It is not known whether this result
holds in general.
In Section~\ref{sec5}, we apply Theorem~\ref
{theoremsemialgebraicdescription}
to derive the topological boundary and the algebraic boundary of
$\mathcal{M}$.
Also, using what follows in Section~\ref{sec4},
we shall see how these boundaries are detected by the EM algorithm.

\section{Decomposing the variety of EM fixed points}
\label{sec4}

After this in-depth study of the geometry of our model,
we now return to the fixed points of Expectation--Maximization on
$\mathcal{M}$.
We fix the polynomial ring $\Q[A,R,B]$ in
$mr+mn+rn$ indeterminates $a_{ik}$, $r_{ij}$ and $b_{kj}$.
Let $\mathcal{F}$ denote the ideal generated by the
entries of the matrices
$A\star(R\cdot B^T) $
and $ B\star(A^T\cdot R) $ in
(\ref{EMfixed2}).
Also, let $\mathcal{C}$ denote the ideal generated by
the entries of $R \cdot B^T$ and $A^T \cdot R$.
Thus, $\mathcal{F}$ is generated by
$mr+rn$ cubics, $\mathcal{C}$ is generated
by $mr + rn$ quadrics, and we have
the inclusion \mbox{$\mathcal{F} \subset\mathcal{C}$}.
By Theorem~\ref{thmEMfixedTheta},
the variety $V(\mathcal{C})$ consists of those parameters $A, R, B$
that correspond to critical points for the log-likelihood function
$\ell_U$,
while the variety $V(\mathcal{F})$ encompasses all the fixed points of the
EM algorithm. We are interested in the
irreducible components of the varieties $V(\mathcal{F})$ and
$V(\mathcal{C})$.
These are the zero sets
of the minimal primes of $\mathcal{F}$ and $\mathcal{C}$,
respectively. More precisely, if $\mathcal{F}$ has minimal primes
$\mathcal{F}_1, \mathcal{F}_2,\ldots, \mathcal{F}_N$, then
$V(\mathcal{F}_i)$ are the irreducible components of $V(\mathcal
{F})$, and $V(\mathcal{F}) = \bigcup_iV(\mathcal{F}_i)$.

Recall that the matrix $R $ represents the gradient of the
log-likelihood function~\mbox{$\ell_U$}, that is,
%
\begin{equation}
\label{eqwhatisR} r_{ij}   =   u_{++} - \frac{u_{ij}}{p_{ij}}
  =   u_{++} - \frac{u_{ij}}{\sum_k a_{ik}\lambda_kb_{kj}}.
\end{equation}
The set of EM-fixed points corresponding to
a data matrix $U\in\mathbb{N}^{m\times n}$ is defined by the
ideal $ \mathcal{F}' \subset\mathbb{Q}[A, B, \Lambda]$ that
is obtained from $\mathcal{F}$ by substituting
(\ref{eqwhatisR}), clearing denominators, and saturating.
Note that $V(\mathcal{F}' ) = \bigcup_iV(\mathcal{F}_i' )$. So,
studying the minimal primes $\mathcal{F}_i$ will help us study the
fixed points of EM.
A big advantage of considering $\mathcal{F}$ rather than $\mathcal
{F}'$ is
that $\mathcal{F}$ is much simpler. Also, it does not depend on the
data $U$.
This allows a lot of the work in exact
MLE using algebraic methods
(as in Example~\ref{exUabmatrix})
to be done in a preprocessing stage.

There are two important points we wish to make in this section:
\begin{enumerate}[2.]
\item[1.] the minimal primes of $\mathcal{F}$ have interesting
statistical interpretations, and

\item[2.] the nontrivial boundaries of the mixture model $\mathcal
{M}$ can be detected from this.
\end{enumerate}
We shall explain these points by working out two cases that are larger
than Example~\ref{exEM22}.

Example~\ref{exEM22} showed that $\mathcal{F}$ is
not radical but has embedded components.
Here, we focus on the minimal primes $\mathcal{F}_i$ of $\mathcal{F}$,
as these correspond to geometric components of $V(\mathcal{F})$.
If $\mathcal{F}_i$ is also a minimal prime of $\mathcal{C}$ then
$\mathcal{F}_i$ is a \textit{critical} prime\vadjust{\goodbreak} of $\mathcal{F}$.
Not every minimal prime of $\mathcal{C}$
is a minimal prime of $\mathcal{F}$. For instance, for
$m=n=2, r=1$, the ideal
$ \mathcal{C} $ is the intersection of the
first prime in Example~\ref{exEM22} and
$\langle a_{11}, a_{21}, b_{11}, b_{12} \rangle$.
The latter is not minimal over $\mathcal{F}$.
We now generalize this example:

\begin{proposition} \label{propquivercycle}
The ideal $\mathcal{C}$ has precisely $r+1$ minimal primes, indexed
by $k=1,\ldots,r+1  $:
\begin{eqnarray*}
&& \mathcal{C} + \langle k \mbox{-minors of }A \rangle
+ \bigl\langle  (m-k+2) \mbox{-minors of }R \bigr\rangle
\\
&&\qquad{}+ \bigl\langle (n-m+k) \mbox{-minors of }B \bigr\rangle
\qquad
\mbox{if }m \leq n,
\\
&& \mathcal{C} + \bigl\langle (m-n+k) \mbox{-minors of }A
\bigr\rangle + \bigl\langle  (n-k+2) \mbox{-minors of }R \bigr\rangle
\\
&&\qquad{}+ \langle k \mbox{-minors of }B \rangle
\qquad\mbox{if }m \geq n.
\end{eqnarray*}
Moreover, the ideal $\mathcal{C}$ is radical, and hence, it equals the
intersection of its minimal primes.
\end{proposition}

We refer to
Example~\ref{exC443Appendix} for an illustration of
Proposition~\ref{propquivercycle}. The proof we give in
Appendix~\ref{appproofs} relies on
methods from representation theory.
The duality relation~\eqref{eqduality} plays an important role.

We now proceed to our case studies of
the minimal primes of the EM fixed ideal~$\mathcal{F}$.

\begin{example} \label{exem33}
Let $m=n=3$ and $r=2$. The ideal $\mathcal{F}$ has $37$ minimal primes,
in six classes.
The first three are the minimal primes of
the critical ideal $\mathcal{C}$, as seen in Proposition~\ref
{propquivercycle}:
\begin{eqnarray*}
I_1 &=& \langle
r_{23}r_{32}-r_{22}r_{33},r_{13}r_{32}-r_{12}r_{33},r_{23}r_{31}-r_{21}r_{33},r_{22}r_{31}-r_{21}r_{32},
\\
&&\hspace*{4pt} r_{13}r_{31}-r_{11}r_{33},r_{12}r_{31}-r_{11}r_{32},r_{13}r_{22}-r_{12}r_{23},r_{13}r_{21}-r_{11}r_{23},
\\
&&\hspace*{4pt} r_{12}r_{21}-r_{11}r_{22},b_{21}r_{31}+b_{22}r_{32}+b_{23}r_{33},b_{11}r_{31}+b_{12}r_{32}+b_{13}r_{33},
\\
&&\hspace*{4pt} b_{21}r_{21}+b_{22}r_{22}+b_{23}r_{23}, b_{11}r_{21}+b_{12}r_{22}+b_{13}r_{23},
\\
&&\hspace*{4pt} a_{12}r_{13}+a_{22}r_{23}+a_{32}r_{33}, a_{11}r_{13}+a_{21}r_{23}+a_{31}r_{33},
\\
&&\hspace*{4pt} a_{12}r_{12}+a_{22}r_{22}+a_{32}r_{32},a_{11}r_{12}+a_{21}r_{22}+a_{31}r_{32},
\\
&&\hspace*{4pt} b_{21}r_{11}+b_{22}r_{12}+b_{23}r_{13},b_{11}r_{11}+b_{12}r_{12}+b_{13}r_{13},
\\
&&\hspace*{75pt}\hspace*{4pt}a_{12}r_{11}+a_{22}r_{21}+a_{32}r_{31}, a_{11}r_{11}+a_{21}r_{21}+a_{31}r_{31}\rangle,
\\
I_2  &=   & \langle
r_{13}r_{22}r_{31}-r_{12}r_{23}r_{31}-r_{13}r_{21}r_{32}+r_{11}
r_{23}r_{32}+r_{12}r_{21}r_{33}-r_{11}r_{22}r_{33},
\\
&&\hspace*{4pt} b_{21}r_{31}+b_{22}r_{32}+b_{23}r_{33},b_{11}r_{31}+b_{12}r_{32}+b_{13}r_{33},
\\
&&\hspace*{4pt}b_{21}r_{21}+b_{22}r_{22}+ b_{23}r_{23}, b_{11}r_{21}+b_{12}r_{22}+b_{13}r_{23},
\\
&&\hspace*{4pt} a_{12}r_{13}+a_{22}r_{23}+a_{32}r_{33},a_{11}r_{13}+a_{21}r_{23}+a_{31}r_{33},
\\
&&\hspace*{4pt}  a_{12}r_{12}+a_{22}r_{22}+a_{32}r_{32},a_{11}r_{12}+a_{21}r_{22}+a_{31}r_{32},
\\
&&\hspace*{4pt} b_{21}r_{11}+b_{22}r_{12}+b_{23}r_{13}, b_{11}r_{11}+b_{12}r_{12}+b_{13}r_{13},
\\
&&\hspace*{4pt} a_{12}r_{11}+a_{22}r_{21}+a_{32}r_{31},a_{11}r_{11}+a_{21}r_{21}+a_{31}r_{31},
\\
&&\hspace*{4pt} b_{13}b_{22}-b_{12}b_{23},b_{13}b_{21}-b_{11}b_{23},b_{12}b_{21}-b_{11}b_{22},
\\
&&\hspace*{97pt}\hspace*{4pt} a_{31}a_{22}-a_{21}a_{32}, a_{31}a_{12}-a_{11}a_{32},a_{21}a_{12}-a_{11}a_{22}
\rangle,
\\
I_3   &= & \langle
a_{11},a_{21},a_{31},a_{12},a_{22},a_{32},b_{11},b_{12},b_{13},b_{21},
b_{22}, b_{23} \rangle.
\end{eqnarray*}
In addition to these three, $\mathcal{F}$ has $12$ noncritical
components like
\begin{eqnarray*}
J_1   &=& \langle a_{11},a_{21},r_{31}, r_{32}, r_{33}, r_{13}r_{22}-r_{12}r_{23},r_{13}r_{21}-r_{11}r_{23},
\\
&&\hspace*{4pt} r_{12}r_{21}-r_{11}r_{22},b_{21}r_{21}+b_{22}r_{22}+b_{23}r_{23},b_{21}r_{11}+b_{22}r_{12}+b_{23}r_{13},
\\
&&\hspace*{83pt}\hspace*{4pt} a_{12}r_{13}+a_{22}r_{23},a_{12}r_{12}+a_{22}r_{22},a_{12}r_{11}+a_{22}r_{21} \rangle,
\end{eqnarray*}
four noncritical components like
\begin{eqnarray*}
J_2  &= & \langle a_{11},a_{21},a_{31},
r_{13}r_{22}r_{31}-r_{12}r_{23}r_{31}-r_{13}r_{21}r_{32}+r_{11}r_{23}r_{32}
\\
&&\hspace*{4pt}{} +r_{12}r_{21}r_{33}-r_{11}r_{22}r_{33},
b_{21}r_{21}+b_{22}r_{22}+b_{23}r_{23},
b_{21}r_{11}
\\
&&\hspace*{4pt}{}+b_{22}r_{12}+b_{23}r_{13},b_{21}r_{31}+b_{22}r_{32}+b_{23}r_{33},
a_{12}r_{13}+a_{22}r_{23}
\\
&&\hspace*{4pt}{}+a_{32}r_{33},
a_{12}r_{12}+a_{22}r_{22}+a_{32}r_{32},
a_{12}r_{11}+a_{22}r_{21}+a_{32}r_{31}
\rangle
\end{eqnarray*}
and $18$ noncritical components like
\begin{eqnarray*}
J_3   &= & \langle a_{11},a_{21},b_{11},b_{12},
r_{33},
r_{13}r_{22}r_{31}-r_{12}r_{23}r_{31}-r_{13}r_{21}r_{32}+r_{11}r_{23}r_{32},
\\
&&\hspace*{4pt} b_{21}r_{31}+b_{22}r_{32},
b_{21}r_{21}+b_{22}r_{22}+b_{23}r_{23},
b_{21}r_{11}+b_{22}r_{12}+b_{23}r_{13},
\\
&&\hspace*{8pt} a_{12}r_{13}+a_{22}r_{23},
a_{12}r_{12}+a_{22}r_{22}+a_{32}r_{32},
a_{12}r_{11}+a_{22}r_{21}+a_{32}r_{31}
\rangle.
\end{eqnarray*}
Each of the $34$ primes $J_1,J_2,J_3$ specifies a face of
the polytope $\Theta$, as it contains
two, three or four of the parameters $a_{ik},b_{kj}$,
and expresses rank constraints on the matrix $R = [r_{ij}]$.
\end{example}

\begin{remark}
Assuming the sample size $u_{++}$ to be known, we can recover the data
matrix $U$
from the gradient $R$ using the formula $ U = R \star P + u_{++} P$.
In coordinates, this says
\[
u_{ij}   =   (r_{ij} + u_{++}) \cdot
p_{ij} \qquad\mbox{for }   i \in[m] ,j \in[n]. %
\]
This formula is obtained by rewriting (\ref{eqwhatisR}).
Hence, $r_{ij} = 0$ holds if and only
if $ p_{ij} = u_{ij} / u_{++}$. This can be rephrased as follows.
If a minimal prime of $\mathcal{F}$ contains the unknown $r_{ij}$, then
the corresponding fixed points of the EM algorithm maintain the cell
entry $u_{ij}$
from the data.
\end{remark}

With this, we can now understand the meaning of the
various components in Example~\ref{exem33}.
The prime $I_1$ parametrizes
critical points $P$ of rank $2$. This represents the behavior of the EM
algorithm
when run with random starting parameters in the\vadjust{\goodbreak} interior of $\Theta$.
For special data $U$, the MLE will be a rank $1$ matrix, and
such cases are captured by the critical component $I_2$.
The components $I_3$ and $J_2$ can be disregarded because
each of them contains a column of $A$. This would force the entries of
that column to sum to $0$, which is impossible in $\Theta$.

The components $J_1$ and $J_3$ describe interesting
scenarios that are realized by starting the EM algorithm with
parameters on
the boundary of the polytope $\Theta$. On the components
$J_1$, the EM algorithm produces an estimate
that maintains one of the rows or columns from the data $U$,
and it replaces the remaining table of format $2 \times3$
or $3 \times2$ by its MLE of rank $1$.
This process amounts to fitting
a context specific independence (CSI)
model to the data. Following Georgi and Schliep \cite{GS}, CSI means that
independence holds only for some values of the involved variables.
Namely,
$J_1$ expresses the constraint that
$X$ is independent of $Y$ given that $Y$ is either $1$ or $2$.
Finally, on the components $J_3$, we have $\operatorname{rank}(A) = \operatorname{rank}(B) = 2$
and $r_{ij} = 0$ for one cell entry $(i,j)$. 

\begin{definition}
Let $\mathcal{F} = \langle A\star(R\cdot
B^T) ,
B\star(A^T\cdot R) \rangle$ be the ideal of EM fixed points.
A minimal prime of $\mathcal{F}$ is
called \textit{relevant} if it contains none of the
$mn$ polynomials $p_{ij} = \sum_{k=1}^r a_{ik} b_{kj}$.
\end{definition}

In Example~\ref{exEM22}, only the first minimal prime is relevant. In
Example~\ref{exem33}, all minimal primes besides $I_3$ are relevant.
Restricting to the relevant minimal primes is justified because the EM
algorithm never outputs a matrix containing zeros for positive starting
data. Note also that
the $p_{ij} $ appear in the denominators in
the expressions (\ref{eqwithdenom}) that were used in our derivation
of $\mathcal{F}$.

Our main result in this section is
the computation in Theorem~\ref{thmem44}.
We provide a census of
EM fixed points for $4 \times 4$-matrices
of rank $r = 3$. This is the smallest case where
rank can differ from nonnegative rank,
and the boundary hypersurfaces~(\ref{eqsixthree}) appear.

\begin{theorem}
\label{thmem44}
Let $m=n=4$ and $r=3$. The radical of the EM fixed point ideal
$\mathcal{F}$ has
49,000 relevant primes.
These come in $108$ symmetry classes,
listed in Table~\ref{tableminimalprimes44}.
\end{theorem}

\begin{table}
\tabcolsep=0pt
\caption{Minimal primes of the
EM fixed ideal $\mathcal{F}$ for $4 \times 4$-matrices
of rank $3$}\label{tableminimalprimes44}
\begin{tabular*}{\tablewidth}{@{\extracolsep{\fill}}@{}lcccd{4.0}cccccd{3.0}@{}}
\hline
\textbf{Set} $\bolds{S}$ & \multicolumn{1}{c}{$\bolds{|S|}$} & \multicolumn{1}{c}{$\bolds{a}$\textbf{'s}} & \multicolumn{1}{c}{$\bolds{b}$\textbf{'s}} &
\multicolumn{1}{c}{\textbf{deg}} & \multicolumn{1}{c}{\textbf{codim}} & \multicolumn{1}{c}{\textbf{rA}} & \multicolumn{1}{c}{\textbf{rB}} &
\multicolumn{1}{c}{\textbf{rR}} & \multicolumn{1}{c}{\textbf{rP}} & \multicolumn{1}{c}{$\bolds{|\mathrm{orbit}|}$} \\
\hline
$ \varnothing$ & 0 & 0 & 0 & 1 & 24 & 0 & 0 & 4 & 0 & 1\\
& 0 & 0 & 0 & 1630 & 19 & 1 & 1 & 3 & 1 & 1\\
& 0 & 0 & 0 & 3491 & 16 & 2 & 2 & 2 & 2 & 1\\
& 0 & 0 & 0 & 245 & 15 & 3 & 3 & 1 & 3 & 1
\\
$\{a_{11}\}$ & 1 & 1 & 0 & 245 & 16 & 3 & 3 & 1 & 3 & 24\\
& 1 & 1 & 0 & 3491 & 17 & 2 & 2 & 2 & 2 & 24
\\
$\{a_{11},a_{21}\}$ & 2 & 2 & 0 & 20 & 17 & 3 & 3 & 1 & 3 & 36\\
& 2 & 2 & 0 & 245 & 17 & 3 & 3 & 1 & 3 & 36\\
& 2 & 2 & 0 & 1460 & 17 & 2 & 3 & 2 & 2 & 36
\\
$\{a_{11},a_{21},a_{31}\}$ & 3 & 3 & 0 & 53 & 17 & 3 & 3 & 1 & 3 & 24\\
& 3 & 3 & 0 & 188 & 17 & 2 & 3 & 2 & 2 & 24
\\
$*\{a_{11},a_{21},b_{11},b_{12}\}*$ & 4 & 2 & 2 & 245 & 19 & 3 & 3 & 1& 3 & 108\\
& 4 & 2 & 2 & 20 & 19 & 3 & 3 & 1 & 3 & \multicolumn{1}{c@{}}{$108\times 2$} \\
& 4 & 2 & 2 & 1460 & 19 & 2 & 3 & 2 & 2 & \multicolumn{1}{c@{}}{$108\times 2$}\\
& 4 & 2 & 2 & 2370 & 20 & 2 & 2 & 3 & 2 & 108\\
& 4 & 2 & 2 & 240 & 19 & 3 & 3 & 2 & 3 & 108
\\
$\{a_{11},a_{21},b_{21},b_{22}\}$ & 4 & 2 & 2 & 825 & 18 & 3 & 3 & 2 & 3 & 216\\
$\{a_{11},a_{21},a_{31},a_{41}\}$ & 4 & 4 & 0 & 689 & 16 & 2 & 3 & 2 & 2 & 6\\
& 4 & 4 & 0 & 474 & 17 & 1 & 2 & 3 & 1 & 6\\
$\{a_{11},a_{21},a_{12},a_{22}\}$ & 4 & 4 & 0 & 592 & 17 & 2 & 3 & 2 &
2 & 36\\
& 4 & 4 & 0 & 9 & 17 & 3 & 3 & 1 & 3 & 36\\
$\{a_{11},a_{21},a_{32},a_{42}\}$ & 4 & 4 & 0 & 20 & 19 & 3 & 3 & 1 & 3 & \multicolumn{1}{c@{}}{$36\times 2$}\\
& 4 & 4 & 0 & 245 & 19 & 3 & 3 & 1 & 3 & 36\\
& 4 & 4 & 0 & 400 & 18 & 2 & 3 & 2 & 2 & 36\\
$\{a_{11},a_{21},a_{31},b_{11},b_{12}\}$ & 5 & 3 & 2 & 474 & 20 & 2 & 2
& 3 & 2 & 144\\
& 5 & 3 & 2 & 188 & 19 & 2 & 3 & 2 & 2 & 144\\
& 5 & 3 & 2 & 448 & 19 & 3 & 3 & 2 & 3 & 144\\
& 5 & 3 & 2 & 53 & 19 & 3 & 3 & 1 & 3 & 144\\
$\{a_{11},a_{21},a_{31},b_{21},b_{22}\}$ & 5 & 3 & 2 & 125 & 18 & 3 & 3
& 2 & 3 & 288\\
$\{a_{11},a_{21},a_{32},a_{42},b_{31}\}$ & 5 & 4 & 1 & 723 & 19 & 3 & 3
& 2 & 3 & 144\\
$\{a_{11},a_{21},a_{31},b_{11},b_{12},b_{13}\}$ & 6 & 3 & 3 & 689 & 19
& 3 & 3 & 2 & 3 & 48\\
& 6 & 3 & 3 & 474 & 20 & 2 & 2 & 3 & 2 & 48\\
$\{a_{11},a_{21},a_{31},b_{21},b_{22},b_{23}\}$ & 6 & 3 & 3 & 21 & 18 &
3 & 3 & 2 & 3 & 96\\
$\{a_{11},a_{21},a_{32},b_{11},b_{12},b_{33}\}$ & 6 & 3 & 3 & 2785 & 20
& 3 & 3 & 3 & 3 & 864\\
$*\{a_{11},a_{22},a_{33},b_{11},b_{22},b_{33}\}*$ & 6 & 3 & 3 & 9016 &
21 & 3 & 3 & 4 & 3 & 576\\
& 6 & 3 & 3 & 245 & 21 & 3 & 3 & 1 & 3 & 576\\
$\{a_{11},a_{21},a_{31},a_{41},b_{21},b_{22}\}$ & 6 & 4 & 2 & 265 & 17
& 2 & 3 & 2 & 2 & 72\\
$\{a_{11},a_{21},a_{12},a_{22},b_{11},b_{12}\}$ & 6 & 4 & 2 & 592 & 19
& 2 & 3 & 2 & 2 & 432\\
& 6 & 4 & 2 & 9 & 19 & 3 & 3 & 1 & 3 & 432\\
& 6 & 4 & 2 & 104 & 19 & 3 & 3 & 2 & 3 & 432\\
$\{a_{11},a_{21},a_{32},a_{42},b_{11},b_{12}\}$ & 6 & 4 & 2 & 825 & 20
& 3 & 3 & 2 & 3 & 432\\
& 6 & 4 & 2 & 100 & 20 & 3 & 3 & 2 & 3 & 432\\
& 6 & 4 & 2 & 400 & 20 & 2 & 3 & 2 & 2 & 432\\
$\{a_{11},a_{21},a_{32},a_{42},b_{31},b_{32}\}$ & 6 & 4 & 2 & 301 & 19
& 3 & 3 & 2 & 3 & 216\\
\hline
\end{tabular*}
\end{table}
\setcounter{table}{1}
\begin{table}
\tabcolsep=0pt
\caption{(Continued)}
\begin{tabular*}{\tablewidth}{@{\extracolsep{\fill}}@{}lcccd{4.0}cccccd{4.0}@{}}
\hline
\textbf{Set} $\bolds{S}$ & \multicolumn{1}{c}{$\bolds{|S|}$} & \multicolumn{1}{c}{$\bolds{a}$\textbf{'s}} & \multicolumn{1}{c}{$\bolds{b}$\textbf{'s}} &
\multicolumn{1}{c}{\textbf{deg}} & \multicolumn{1}{c}{\textbf{codim}} & \multicolumn{1}{c}{\textbf{rA}} & \multicolumn{1}{c}{\textbf{rB}} &
\multicolumn{1}{c}{\textbf{rR}} & \multicolumn{1}{c}{\textbf{rP}} & \multicolumn{1}{c}{$\bolds{|\mathrm{orbit}|}$} \\
\hline
$\{a_{11},a_{21},a_{31},a_{41},a_{12},a_{22}\}$ & 6 & 6 & 0 & 265 & 17
& 2 & 3 & 2 & 2 & 72\\
$\{a_{11},a_{21},a_{31},a_{12},a_{22},a_{32}\}$ & 6 & 6 & 0 & 35 & 16 &
2 & 3 & 2 & 2 & 24\\
$\{a_{11},a_{21},a_{12},a_{22},a_{33},a_{43}\}$ & 6 & 6 & 0 & 180 & 18
& 2 & 3 & 2 & 2 & 36\\
& 6 & 6 & 0 & 9 & 19 & 3 & 3 & 1 & 3 & 36\\
$\{a_{11},a_{21},a_{31},a_{41},b_{21},b_{22},b_{23}\}$ & 7 & 4 & 3 & 35
& 17 & 2 & 3 & 2 & 2 & 48\\
$\{a_{11},a_{21},a_{31},a_{42},b_{11},b_{12},b_{33}\}$ & 7 & 4 & 3 &
557 & 20 & 3 & 3 & 3 & 3 & 576\\
$\{a_{11},a_{21},a_{12},a_{22},b_{11},b_{12},b_{13}\}$ & 7 & 4 & 3 &
191 & 19 & 3 & 3 & 2 & 3 & 288\\
$\{a_{11},a_{21},a_{32},a_{42},b_{11},b_{12},b_{13}\}$ & 7 & 4 & 3 &
140 & 20 & 3 & 3 & 2 & 3 & 288\\
& 7 & 4 & 3 & 125 & 20 & 3 & 3 & 2 & 3 & 288\\
$\{a_{11},a_{21},a_{32},a_{42},b_{11},b_{12},b_{33}\}$ & 7 & 4 & 3 &
835 & 20 & 3 & 3 & 3 & 3 & 864\\
$\{a_{11},a_{21},a_{32},a_{42},b_{31},b_{32},b_{33}\}$ & 7 & 4 & 3 & 49
& 19 & 3 & 3 & 2 & 3 & 144\\
$*\{a_{11},a_{21},a_{32},a_{43},b_{11},b_{22},b_{33}\}*$ & 7 & 4 & 3 &
3087 & 21 & 3 & 3 & 4 & 3 & 1728\\
$\{a_{11},a_{21},a_{31},a_{12},a_{22},b_{21},b_{22}\}$ & 7 & 5 & 2 & 31
& 19 & 3 & 3 & 2 & 3 & 864\\
$\{a_{11},a_{21},a_{31},a_{12},a_{42},b_{11},b_{12}\}$ & 7 & 5 & 2 &
225 & 20 & 3 & 3 & 2 & 3 & 864\\
$\{a_{11},a_{21},a_{12},a_{32},a_{43},b_{11},b_{22}\}$ & 7 & 5 & 2 &
1193 & 21 & 3 & 3 & 3 & 3 & 1728\\
$\{a_{11},a_{21},a_{31},a_{41},b_{21},b_{22},b_{23},b_{24}\}$ & 8 & 4 &
4 & 85 & 15 & 2 & 2 & 3 & 1 & 6\\
$\{a_{11},a_{21},a_{31},a_{41},b_{21},b_{22},b_{33},b_{34}\}$ & 8 & 4 &
4 & 81 & 18 & 2 & 3 & 2 & 2 & 36\\
$\{a_{11},a_{21},a_{31},a_{42},b_{11},b_{12},b_{13},b_{34}\}$ & 8 & 4 &
4 & 557 & 20 & 3 & 3 & 3 & 3 & 96\\
$\{a_{11},a_{21},a_{31},a_{42},b_{11},b_{12},b_{33},b_{34}\}$ & 8 & 4 &
4 & 167 & 20 & 3 & 3 & 3 & 3 & 288\\
$\{a_{11},a_{21},a_{12},a_{22},b_{11},b_{12},b_{21},b_{22}\}$ & 8 & 4 &
4 & 850 & 20 & 2 & 2 & 3 & 2 & 108\\
& 8 & 4 & 4 & 45 & 19 & 3 & 3 & 2 & 3 & 108\\
$\{a_{11},a_{21},a_{12},a_{22},b_{11},b_{12},b_{23},b_{24}\}$ & 8 & 4 &
4 & 9 & 21 & 3 & 3 & 1 & 3 & 216\\
& 8 & 4 & 4 & 1024 & 21 & 3 & 2 & 3 & 2 & 216\\
& 8 & 4 & 4 & 104 & 21 & 3 & 3 & 2 & 3 & \multicolumn{1}{c@{}}{$216\times 2$}\\
& 8 & 4 & 4 & 592 & 21 & 2 & 3 & 2 & 2 & 216\\
$\{a_{11},a_{21},a_{12},a_{32},b_{11},b_{12},b_{21},b_{23}\}$ & 8 & 4 &
4 & 2121 & 21 & 3 & 3 & 3 & 3 & 1728\\
$\{a_{11},a_{21},a_{12},a_{32},b_{11},b_{12},b_{23},b_{24}\}$ & 8 & 4 &
4 & 2125 & 21 & 3 & 3 & 3 & 3 & 864\\
$\{a_{11},a_{21},a_{32},a_{42},b_{11},b_{12},b_{23},b_{24}\}$ & 8 & 4 &
4 & 2125 & 21 & 3 & 3 & 3 & 3 & 108\\
$\{a_{11},a_{21},a_{32},a_{42},b_{11},b_{12},b_{33},b_{34}\}$ & 8 & 4 &
4 & 265 & 20 & 3 & 3 & 3 & 3 & 216\\
$\{a_{11},a_{21},a_{32},a_{43},b_{11},b_{12},b_{23},b_{34}\}$ & 8 & 4 &
4 & 2205 & 21 & 3 & 3 & 4 & 3 & 432\\
$\{a_{11},a_{21},a_{32},a_{43},b_{11},b_{22},b_{23},b_{34}\}$ & 8 & 4 &
4 & 1029 & 21 & 3 & 3 & 4 & 3 & 864\\
$\{a_{11},a_{21},a_{31},a_{12},a_{22},b_{21},b_{22},b_{23}\}$ & 8 & 5 &
3 & 35 & 19 & 3 & 3 & 2 & 3 & 576\\
$\{a_{11},a_{21},a_{31},a_{12},a_{42},b_{11},b_{12},b_{13}\}$ & 8 & 5 &
3 & 265 & 20 & 3 & 3 & 2 & 3 & 576\\
$\{a_{11},a_{21},a_{12},a_{32},a_{43},b_{11},b_{12},b_{23}\}$ & 8 & 5 &
3 & 1185 & 21 & 3 & 3 & 3 & 3 & 3456\\
$\{a_{11},a_{21},a_{31},a_{41},a_{12},a_{22},b_{21},b_{22}\}$ & 8 & 6 &
2 & 425 & 18 & 2 & 3 & 3 & 2 & 432\\
$\{a_{11},a_{21},a_{12},a_{22},a_{33},a_{43},b_{11},b_{12}\}$ & 8 & 6 &
2 & 180 & 20 & 2 & 3 & 2 & 2 & 432\\
& 8 & 6 & 2 & 45 & 20 & 3 & 3 & 2 & 3 & 432\\
$\{a_{11},a_{21},a_{31},a_{41},a_{12},a_{22},a_{32},a_{42}\}$ & 8 & 8 &
0 & 85 & 15 & 1 & 3 & 3 & 1 & 6\\
$\{a_{11},a_{21},a_{31},a_{41},a_{12},a_{22},a_{33},a_{43}\}$ & 8 & 8 &
0 & 81 & 18 & 2 & 3 & 2 & 2 & 36\\
$\{a_{11},a_{21},a_{31},a_{12},a_{22},b_{11},b_{12},b_{23},b_{24}\}$ &
9 & 5 & 4 & 296 & 21 & 3 & 3 & 3 & 3 & 864\\
& 9 & 5 & 4 & 31 & 21 & 3 & 3 & 2 & 3 & 864\\
$\{a_{11},a_{21},a_{31},a_{12},a_{42},b_{11},b_{12},b_{21},b_{23}\}$ &
9 & 5 & 4 & 425 & 21 & 3 & 3 & 3 & 3 & 3456\\
$\{a_{11},a_{21},a_{31},a_{12},a_{42},b_{11},b_{12},b_{23},b_{24}\}$ &
9 & 5 & 4 & 425 & 21 & 3 & 3 & 3 & 3 & 864\\
$\{a_{11},a_{21},a_{12},a_{22},a_{33},b_{11},b_{12},b_{23},b_{24}\}$ &
9 & 5 & 4 & 839 & 21 & 3 & 3 & 3 & 3 & 432\\
\hline
\end{tabular*}
\end{table}
\setcounter{table}{1}
\begin{table}
\tabcolsep=0pt
\caption{(Continued)}
\begin{tabular*}{\tablewidth}{@{\extracolsep{\fill}}@{}ld{2.0}ccd{4.0}cccccd{4.0}@{}}
\hline
\textbf{Set} $\bolds{S}$ & \multicolumn{1}{c}{$\bolds{|S|}$} & \multicolumn{1}{c}{$\bolds{a}$\textbf{'s}} & \multicolumn{1}{c}{$\bolds{b}$\textbf{'s}} &
\multicolumn{1}{c}{\textbf{deg}} & \multicolumn{1}{c}{\textbf{codim}} & \multicolumn{1}{c}{\textbf{rA}} & \multicolumn{1}{c}{\textbf{rB}} &
\multicolumn{1}{c}{\textbf{rR}} & \multicolumn{1}{c}{\textbf{rP}} & \multicolumn{1}{c}{$\bolds{|\mathrm{orbit}|}$} \\
\hline
$\{a_{11},a_{21},a_{12},a_{32},a_{43},b_{11},b_{12},b_{13},b_{24}\}$ &
9 & 5 & 4 & 237 & 21 & 3 & 3 & 3 & 3 & 1152\\
$\{a_{11},a_{21},a_{12},a_{32},a_{43},b_{11},b_{12},b_{23},b_{24}\}$ &
9 & 5 & 4 & 875 & 21 & 3 & 3 & 3 & 3 & 864\\
$\{a_{11},a_{21},a_{31},a_{41},a_{12},a_{22},b_{21},b_{22},b_{23}\}$ &
9 & 6 & 3 & 85 & 18 & 2 & 3 & 3 & 2 & 288\\
$\{a_{11},a_{21},a_{31},a_{12},a_{22},a_{43},b_{11},b_{12},b_{23}\}$ &
9 & 6 & 3 & 163 & 21 & 3 & 3 & 3 & 3 & 1728\\
$\{a_{11},a_{21},a_{12},a_{22},a_{33},a_{43},b_{11},b_{12},b_{13}\}$ &
9 & 6 & 3 & 63 & 20 & 3 & 3 & 2 & 3 & 288\\
$\{a_{11},a_{21},a_{31},a_{41},a_{12},a_{22},a_{32},b_{21},b_{22}\}$ &
9 & 7 & 2 & 85 & 18 & 2 & 3 & 3 & 2 & 288\\
$\{
a_{11},a_{21},a_{31},a_{12},a_{22},b_{11},b_{12},b_{13},b_{21},b_{24}\}
$ & 10 & 5 & 5 & 425 & 21 & 3 & 3 & 3 & 3 & 1728\\
$\{
a_{11},a_{21},a_{31},a_{12},a_{22},b_{11},b_{12},b_{21},b_{22},b_{23}\}
$ & 10 & 5 & 5 & 85 & 20 & 3 & 3 & 3 & 3 & 864\\
$\{
a_{11},a_{21},a_{31},a_{12},a_{42},b_{11},b_{12},b_{13},b_{21},b_{24}\}
$ & 10 & 5 & 5 & 425 & 21 & 3 & 3 & 3 & 3 & 864\\
$\{
a_{11},a_{21},a_{31},a_{12},a_{42},b_{11},b_{12},b_{21},b_{23},b_{24}\}
$ & 10 & 5 & 5 & 85 & 21 & 3 & 3 & 3 & 3 & 864\\
$\{
a_{11},a_{21},a_{31},a_{12},a_{22},a_{32},b_{11},b_{12},b_{21},b_{22}\}
$ & 10 & 6 & 4 & 85 & 19 & 2 & 3 & 3 & 2 & 144\\
$\{
a_{11},a_{21},a_{31},a_{12},a_{22},a_{42},b_{11},b_{12},b_{21},b_{23}\}
$ & 10 & 6 & 4 & 85 & 21 & 3 & 3 & 3 & 3 & 1728\\
$\{
a_{11},a_{21},a_{31},a_{12},a_{22},a_{42},b_{11},b_{12},b_{23},b_{24}\}
$ & 10 & 6 & 4 & 85 & 21 & 3 & 3 & 3 & 3 & 432\\
$\{
a_{11},a_{21},a_{31},a_{12},a_{22},a_{43},b_{11},b_{12},b_{13},b_{24}\}
$ & 10 & 6 & 4 & 237 & 21 & 3 & 3 & 3 & 3 & 576\\
$\{
a_{11},a_{21},a_{31},a_{12},a_{22},a_{43},b_{11},b_{12},b_{23},b_{24}\}
$ & 10 & 6 & 4 & 175 & 21 & 3 & 3 & 3 & 3 & 864\\
$\{
a_{11},a_{21},a_{12},a_{22},a_{33},a_{43},b_{11},b_{12},b_{23},b_{24}\}
$ & 10 & 6 & 4 & 225 & 21 & 3 & 3 & 3 & 3 & 216\\
$\{
a_{11},a_{21},a_{31},a_{41},a_{12},a_{22},a_{32},b_{21},b_{22},b_{23}\}
$ & 10 & 7 & 3 & 85 & 18 & 2 & 3 & 3 & 2 & 192\\
$\{
a_{11},a_{21},a_{31},a_{12},a_{22},a_{42},b_{11},b_{12},b_{13},b_{21},b_{24}\} $ & 11 & 6 & 5 & 85 & 21 & 3 & 3 & 3 & 3 & 1728\\
$\{a_{11},a_{21},a_{31},a_{12},a_{22},a_{32}$,\\
\quad $b_{11},b_{12},b_{13},b_{21},b_{22},b_{23}\} $ & 12 & 6 & 6 & 85 & 20 & 2 & 2 & 3 & 2 & 48\\
$\{a_{11},a_{21},a_{31},a_{12},a_{22},a_{42}$,
\\
\quad $b_{11},b_{12},b_{13},b_{21},b_{22},b_{24}\}$ & 12 & 6 & 6 & 85 & 21 & 3 & 3 & 3 & 3 & 432\\
\hline
\end{tabular*}
\end{table}

\begin{pf}
We used an approach that mirrors
the primary decomposition of
binomial ideals \cite{ES}. Recall that the EM fixed point ideal equals
\begin{eqnarray*}
\mathcal F   & =&    \bigl\langle A\star\bigl(R\cdot B^T\bigr),
B\star\bigl(A^T\cdot R\bigr)   \bigr\rangle
\\
& =&    \Biggl\langle  a_{ik} \Biggl(\sum
_{l=1}^n r_{il}b_{kl}\Biggr) ,
  b_{kj}\Biggl(\sum_{l=1}^m
r_{lj} a_{lk}\Biggr)   \dvtx   k\in[r], i\in[m], j\in[n]
\Biggr\rangle.
\end{eqnarray*}
Any prime ideal containing $\mathcal F$ contains either $a_{ik}$ or
$\sum_{l=1}^n r_{il}b_{kl}$ for any $k\in[r]$, $i\in[m]$, and either
$b_{kj}$ or $\sum_{l=1}^m r_{lj} a_{lk}$ for any $k\in[r], j\in[n]$.
We enumerated all primes containing $\mathcal F$ according to the set
$S$ of unknowns $a_{ik}, b_{kj}$ they contain. There are $2^{24}$
subsets and the symmetry group acts on this power
set by replacing $A$ with $B^T$, permuting the rows of $A$,
the columns of $B$, and
the columns of $A$ and
the rows of $B$ simultaneously.
We picked one representative $S$ from each
orbit that is relevant, meaning that we excluded those orbits for which
some $p_{ij} = \sum_{k=1}^r a_{ik} b_{kj}$ lies in
the ideal $\langle S \rangle$.
For each relevant representative $S$, we computed the
\textit{cellular component}
$ \mathcal{F}_S = ( (\mathcal{F} + \langle S \rangle) \dvtx
(\prod S^c)^\infty) $,
where $S^c = \{a_{11}, \ldots, b_{34}\} \setminus S$.
Note that $\mathcal{F}_\varnothing= \mathcal{C} $ is the critical ideal.
We next minimalized our cellular decomposition
by removing all representatives $S$ such that
$ \mathcal{F}_T \subset\mathcal{F}_S$
for some representative $T$ in another orbit.
This led to a list of $76$
orbits, comprising 42,706 ideals
$\mathcal{F}_S$ in total.
For the representative $\mathcal{F}_S$, we computed the set
$\operatorname{Ass}(\mathcal{F}_S)$ of associated primes $P$.
By construction, the sets
$\operatorname{Ass}(\mathcal{F}_S)$ partition
the set of relevant primes of $\mathcal{F}$.
The block sizes $|\mathrm{Ass}(\mathcal{F}_S)|$ range from $1$ to $7$.
Up to symmetry, each prime is uniquely determined
by its attributes in Table~\ref{tableminimalprimes44}.
These are its set $S$, its degree and codimension,
and the ranks
$\operatorname{rA} = \operatorname{rank}(A)$,
$\operatorname{rB} = \operatorname{rank}(B)$,
$\operatorname{rR} = \operatorname{rank}(R)$,
$\operatorname{rP} = \operatorname{rank}(P)$ at a generic point.
Our list starts with the four primes from
coming from $S = \varnothing$. See
Example~\ref{exC443Appendix}.
In each case, the primality of the ideal was verified
using a linear elimination sequence as in \cite{GSS}, Proposition~23(b).
Proofs in {\tt Macaulay2} code are posted on our website.
\end{pf}

Below is the complete list of all $108$ classes
of prime ideals in Theorem~\ref{thmem44}. Three components are marked
with stars.
After the table, we discuss these components in Examples~\ref{exswapping},
\ref{exf32} and
\ref{exf18}.

We illustrate our census of relevant primes
for three sets $S$ that are especially interesting.

\begin{example} \label{exswapping}
Let $S = \{a_{11},a_{21},b_{11},b_{12}\}$.
The cellular component $\mathcal{F}_S$ is the ideal
generated by
$S, \operatorname{det}(R^{34}_{34}),
\operatorname{det}(R)$, and the entries of
the matrices
$B^{23} R^T  ,
B^{1} (  R^T )_{34},
R^T A_{23},
(  R^T  )^{34} A_1$.
In specifying submatrices, upper
indices refer to rows and lower indices refer to columns.
The ideal $\mathcal{F}_S$ is radical with $7$ associated primes,
to be discussed in order of their appearance in Table~\ref
{tableminimalprimes44}.
For instance, the prime (1) below has degree $245$.
The phrase ``Generated by'' is meant modulo~$\mathcal{F}_S$:
\begin{longlist}[(3$'$)]
\item[(1)] Generated by entries of $BR^T,A^TR$,
and $2 \times 2$-minors of $R$. This gives $60$ quadrics.
\item[(2)] Generated by entries of $A^T R,R^{34}$,
and $2 \times 2$-minors of $R,A^{12}_{23}$. This gives
$19$ quadrics.
\item[(2$'$)] Mirror image of (2) under swapping $A$ and $B^T$.
\item[(3)] Generated by entries of $A^T R$,
$2 \times 2$-minors of $A^{12}_{23},R^{34}$,
and $3 \times 3$-minors of $A$, $R^{123},R^{124}$.
This gives $29$ quadrics and $10$ cubics.
\item[(3$'$)] Mirror image of (3) under swapping $A$ and $B^T$.
\item[(4)] Generated by $2 \times 2$-minors
of $A_{23}$ and $B^{23}$. This gives $33$ quadrics and one quartic.
\item[(5)] Generated by entries of $R^{34}_{34}$,
$2 \times 2 $-minors of $R^{12}_{34}, R^{34}_{12}, A^{12}_{23}, B^{23}_{12}$,
and $3 \times 3$-minors of $R$. This gives
$20$ quadrics and $4$ cubics.
\end{longlist}
These primes have the following meaning for
the EM algorithm:
\begin{longlist}[(5)]
\item[(1)] The fixed points $P = \phi(A, R, B)$ given by this
prime ideal are those critical points for the likelihood function $\ell
_U$ for which
the parameters $a_{11}, a_{21}, \break b_{11}, b_{21}$ happen to be $0$.
%
\item[(2)] The fixed points $P = \phi(A, R, B)$ given by this
prime ideal have the last two rows of $P$ fixed and equal to the last
two rows of the data matrix $U$ (divided by the
sample size $u_{++}$). Therefore, the points coming from this ideal are
the maximum likelihood estimates with these eight entries fixed and
which factor so that $a_{11}, a_{21}, b_{11}, b_{21}$ are $0$.
\item[(3)] Since the $3\times3$ minors of $A$ lie in this ideal,
we have rank$(P) \leq2$. Therefore, these fixed points give an
MLE of rank $2$.
This component is the restriction
to $V(\mathcal{F}_S)$ of the generic behavior on the singular locus
of $\mathcal{V}$.
\item[(4)] On this component, the duality relation in (\ref
{eqduality}) fails
since $\operatorname{rank}(P) = 2 $ but $\operatorname{rank}(R) = 3$.
\item[(5)] The fixed points $P = \phi(A, R, B)$ given by this ideal
have the four entries in the last 2 rows and last 2 columns of $P$
fixed and equal to the corresponding entries in $U$ (divided by
$u_{++}$). Therefore, the points coming from this ideal are maximum
likelihood estimates with those four entries fixed, and
parameters $a_{11}, a_{21}, b_{11}, b_{21}$ being $0$.
\end{longlist}
\end{example}

\begin{example} \label{exf32}
Let $S = \{a_{11},a_{21},a_{32},a_{43}, b_{11},b_{22},b_{33} \}$.
The ideal $\mathcal{F}_S$ has codimension $21$, degree $3087$,
and is generated modulo $\langle\mathcal{S} \rangle$
by $20 $ quadrics and two cubics.
To show that $\mathcal{F}_S$ is prime, we use the elimination
method of \cite{GSS}, Proposition 23(b), with the
variable $x_1$ taken successively to be
$ r_{44}$, $r_{43}$, $r_{34}$, $a_{13}$, $r_{21}$, $r_{12}$, $r_{14}$, $r_{33}$, $b_{21}$,
$a_{31}$, $r_{41}$, $a_{21}$, $a_{32}$.

The last elimination ideal is
generated by an irreducible polynomial
of degree~$9$, thus proving primality of $\mathcal{F}_S$.

If we add the relation $P = AB $ to $ \mathcal{F_S}$
and thereafter eliminate $\{A,B,R\}$, we obtain
a prime ideal in $\mathbb{Q}[P]$. That prime ideal
has height one over the
determinantal ideal $\langle \operatorname{det}(P) \rangle$.
Any such prime gives a candidate for a component
in the boundary of our model $\mathcal{M}$.
By matching the set $S$ with the combinatorial analysis in
Section~\ref{sec3}, we see that
Figure~\ref{figcriticalconfigurations}(b) corresponds to $V(S)$.
Hence, by Corollary~\ref{corollaryboundarypointsforrankthreematrices},
this component does in fact contribute to
the boundary $\partial\mathcal{M}$.
This is a special case of Theorem~\ref{thmalgebraicboundary} below;
see equation (\ref{factorizationsofthesecondtype}) in
Example~\ref{exex52}.

This component is the most important one for EM.
It represents the typical behavior when the output of the EM algorithm
is not critical.
In particular, the duality relation (\ref{eqduality}) fails
in the most dramatic form because $\operatorname{rank}(R) = 4$.
As seen in Table~\ref{tablenotcritical},
this failure is still rare $(4.4 \%)$ for $m=n=4$.
For larger matrix sizes, however,
the noncritical behavior occurs with overwhelming~probability.
\end{example}

\begin{example} \label{exf18}
Let $S = \{a_{11},a_{22},a_{33}, b_{11},b_{22},b_{33} \}$.
The computation for the ideal $\mathcal{F}_S$ was the hardest
among all cellular components.
It was found to be radical, with two associated primes
of codimension $21$. The first prime has the largest degree,
namely $9016$, among all entries in Table~\ref{tableminimalprimes44}.
In contrast to Example~\ref{exf32},
the set $S$ cannot contribute to $\partial\mathcal{M}$.
Indeed, for both primes, the elimination ideal in $\mathbb{Q}[P]$
is $\langle\operatorname{det}(P) \rangle$. The degree $9016$ ideal
is the only prime in Table~\ref{tableminimalprimes44}
that has $\operatorname{rank}(R) = 4$ but does not map to the boundary of
the model $\mathcal{M}$.
Starting the EM algorithm
with zero parameters in $S$ generally leads
to the correct MLE.
\end{example}

\section{Algebraic boundaries}\label{sec5}

In Section~\ref{sec3}, we studied the real algebraic geometry
of the mixture model $\mathcal{M}$ for rank three.
In this section, we also fix $r=3$ and focus on the algebraic boundary
of our model.
Our main result in this section is the characterization of its
irreducible components.

\begin{theorem}\label{thmalgebraicboundary}
The algebraic boundary $ \overline{\partial\mathcal{M}} $
is a pure-dimensional reducible variety in $\PP^{mn-1}$. All
irreducible components have dimension $3m+3n-11$ and their number equals
\[
mn + \frac{m(m-1)(m-2)(m+n-6)n(n-1)(n-2)}{4}. %
\]
Besides the $mn$ components $\{ p_{ij} = 0\}$ that come
from $ \partial\Delta_{mn-1}$ there are:
\begin{longlist}[(a)]
\item[(a)]
$36 {m\choose3}{n\choose4}$ components
parametrized by $P = AB$, where
$A$ has three zeros in distinct rows and columns, and
$B$ has four zeros in three rows and distinct columns.
\item[(b)]
$36 {m\choose4} {n\choose3}$ components
parametrized by $P = AB$, where
$A$ has four zeros in three columns and distinct rows,
and $B$ has three zeros in distinct rows and columns.
\end{longlist}
\end{theorem}

This result takes the following specific form in the
first nontrivial case:

\begin{example} \label{exex52}
For $m=n=4$,
the algebraic boundary of
our model $\mathcal{M}$ has $16$ irreducible components $\{p_{ij}=0\}$,
$144$ irreducible components corresponding to factorizations like
%
\begin{eqnarray}\label{factorizationsofthefirsttype}
&& \lleft[\matrix{ p_{11} & p_{12} &
p_{13} & p_{14}
\cr
p_{21} & p_{22} & p_{23} & p_{24}
\cr
p_{31} & p_{32} & p_{33} & p_{34}
\cr
p_{41} & p_{42} & p_{43} & p_{44} }
\rright]
\nonumber\\[-8pt]\\[-8pt]\nonumber
&&\qquad   =   \lleft[\matrix{ 0 & a_{12} & a_{13}
\cr
a_{21} & 0 & a_{23}
\cr
a_{31} & a_{32} & 0
\cr
a_{41} & a_{42} & a_{43} } \rright] \cdot
\lleft[\matrix{ 0 & 0 & b_{13} & b_{14}
\cr
b_{21} & b_{22} & 0 & b_{24}
\cr
b_{31} & b_{32} & b_{33} & 0 } \rright],
\end{eqnarray}
and $144$ irreducible components
that are transpose to those in (\ref
{factorizationsofthefirsttype}), that is,
%
\begin{eqnarray}
\label{factorizationsofthesecondtype}
&& \lleft[\matrix{ p_{11} & p_{12} &
p_{13} & p_{14}
\cr
p_{21} & p_{22} & p_{23} & p_{24}
\cr
p_{31} & p_{32} & p_{33} & p_{34}
\cr
p_{41} & p_{42} & p_{43} & p_{44} }
\rright]
\nonumber\\[-8pt]\\[-8pt]\nonumber
&&\qquad  =   \lleft[\matrix{ 0 & a_{12} & a_{13}
\cr
0 & a_{22} & a_{23}
\cr
a_{31} & 0 & a_{33}
\cr
a_{41} & a_{42} & 0 } \rright] \cdot\lleft[\matrix{ 0 &
b_{12} & b_{13} & b_{14}
\cr
b_{21} & 0 & b_{23} & b_{24}
\cr
b_{31} & b_{32} & 0 & b_{34} } \rright].
\end{eqnarray}
The prime ideal of each component is generated by
the determinant and four polynomials of degree six.
These are the maximal minors of a $ 4 \times5$-matrix.
For the component~(\ref{factorizationsofthesecondtype}),
this can be chosen~as
%
\begin{eqnarray}\label{Concamatrix}
&& \lleft[\matrix{  p_{11} & p_{12} & p_{13} & p_{14} &0
\cr
p_{21} & p_{22} & p_{23} & p_{24}&0
\cr p_{31} & p_{32} & p_{33} & p_{34} &p_{33} (p_{11}p_{22}-p_{12}p_{21})
\cr
p_{41} & p_{42} & p_{43} & p_{44}  & p_{41} (p_{12} p_{23} - p_{13}p_{22} ) +p_{43} (p_{11}p_{22}-p_{12}p_{21})} \rright].\hspace*{-10pt}\nonumber
\\
\end{eqnarray}
This matrix representation was suggested to us
by Aldo Conca and Matteo Varbaro. 
\end{example}

We begin by resolving a problem that was stated
in~\cite{HRS}, Section~5, and~\cite{HS}, Example~2.13.

\begin{proposition} \label{prop633}
The ML degree of each variety
{(\ref{factorizationsofthefirsttype})}
in the algebraic boundary $\overline{\partial\mathcal{M}}$ is $ 633$.
\end{proposition}

Proposition~\ref{prop633} is a first step towards
deriving an exact representation of the MLE function
$U \mapsto\widehat P$
for our model $ \mathcal{M} ={}$\includegraphics{1282i02.eps}.
As highlighted in Table~\ref{tablenotcritical}, the MLE $\widehat P$ typically
lies on the boundary $\partial\mathcal{M}$. We now know that this
boundary has
$304 = 16+144+144$ strata $X_1, X_2, \ldots, X_{304}$.
If $\widehat P$ lies on exactly one of the strata (\ref
{factorizationsofthefirsttype}) or (\ref
{factorizationsofthesecondtype}),
then we can expect the coordinates of $\widehat P$ to be
algebraic numbers of degree $633$ over the rationals $\mathbb{Q}$.
This is the content of Proposition~\ref{prop633}.
By \cite{HRS}, Theorem 1.1,
the degree of $\widehat P$ over $\mathbb{Q}$ is only $191$ if
$\widehat P$ happens to lie
in the interior of $\mathcal{M}$.

In order to complete the exact analysis of MLE for the $4 \times4$-model,
we also need to determine which intersections $X_{i_1} \cap\cdots\cap X_{i_s}$
are nonempty on $\partial\mathcal{M}$. For each
such nonempty stratum, we would then need to compute its ML degree.
This is a challenge left for a future project.

\begin{pf*}{Proof of Theorem~\ref{thmalgebraicboundary}}
By Corollary~\ref{corollaryboundarypointsforrankthreematrices},
an $m\times n$ matrix $P$ of rank $3$ without zero entries lies on
$\partial\mathcal{M}^{m \times n}_3$ if and only if all
triangles $\Delta$ with $\mathcal{B} \subseteq\Delta\subseteq
\mathcal{A}$ contain an edge of $\mathcal{B}$ on one of its edges and
a vertex of $\mathcal{B}$ on all other edges, or one of its vertices
coincides with a vertex of $\mathcal{A}$ and all other edges contain a
vertex of $\mathcal{B}$. We will write down these conditions algebraically.

The columns of $A$ correspond to the vertices of $\Delta$, and
the columns of $B$ correspond to the convex combinations of the
vertices of $\Delta$ that give
the columns of $P=AB$. If a vertex of $\Delta$ and a vertex of
$\mathcal{A}$ coincide, then the corresponding column of
$A$ has two $0$'s. Otherwise the corresponding column of $A$ has one
$0$. If a vertex of $\mathcal{B}$ lies on an edge of $\Delta$, then
one entry of $B$ is zero.

We can freely permute the columns of the left
$m \times3$ matrix $A$ of a factoriza\-tion---this corresponds to
permuting the rows of the corresponding right $ 3 \times n$
matrix $B$. Thus we can assume that the first column contains two $0$'s
and/or the rest of the $0$'s appear in the increasing order.

In the first case, there are ${m\choose3}$ possibilities for choosing
the three rows of $A$ containing $0$'s, there are $3$ choices for the row
of $B$ with two $0$'s, ${n\choose2}$ possibilities for choosing the
positions for the two $0$'s,
and $(n-2)(n-3)$ possibilities for choosing the positions of the $0$'s
in the other two rows of $B$.
In the second case, there are ${m\choose2}$ possibilities for choosing
the $0$'s in the first column of $A$ and ${m-2\choose2}$ choices for
the positions of the $0$'s in other columns. There\vspace*{1pt} are
${n\choose3}$ choices for the columns of $B$ containing $0$'s and $3!$
choices for the positions of the $0$'s in these columns.
\end{pf*}

The prime ideal in (\ref{Concamatrix}) can be found and
verified by direct computation, for example,~by using the software {\tt
Macaulay2} \cite{M2}.
For general values of $m$ and $n$,
the prime ideal of an irreducible boundary component
is generated by quartics and sextics that generalize those in
Example~\ref{exex52}.
The following theorem was stated as a conjecture in the original
December 2013 version of this paper.
That conjecture was proved in April 2014 by
Eggermont, Horobe\c{t} and Kubjas \cite{EHK}.

\begin{theorem}[(Eggermont, Horobe\c{t} and Kubjas)] \label{conjmingens}
Let $m \geq4, n \geq3$ and consider the irreducible component
of $\overline{ \partial\mathcal{M}}$ in
Theorem~\ref{thmalgebraicboundary}\textup{(b)}.
The prime ideal of this component is minimally generated
by ${m\choose4}{n\choose4}$ quartics, namely the $4\times4$-minors
of~$P$, and by ${n\choose3}$ sextics that are indexed by subsets $\{
i,j,k\}$ of $\{1,2,\ldots,n\}$.
These form a Gr\"obner basis with respect to graded reverse
lexicographic order. The sextic indexed
by $\{i,j,k\}$ is homogeneous of degree
$e_1+e_2+e_3+e_i+e_j+e_k$ in the column
grading by $\mathbb{Z}^n$ and homogeneous of degree
$2e_1+2e_2+e_3+e_4$ in the row grading
by $\mathbb{Z}^m$.
\end{theorem}

The row and column gradings of the polynomial ring $\mathbb{Q}[P]$
are given by $\operatorname{deg}(p_{ij}) = e_i$
and $\operatorname{deg}(p_{ij}) = e_j$
where $e_i$ and $e_j$ are unit vectors in $\mathbb{Z}^m$ and
$\mathbb{Z}^n$, respectively.

\begin{example}
If $m = 5$ and $n = 6$, then our component is given by the parametrization
\begin{eqnarray*}
&& \lleft[\matrix{ p_{11} & p_{12} & p_{13} & p_{14} & p_{15} & p_{16}
\cr
p_{21} & p_{22} & p_{23} & p_{24} &
p_{25} & p_{26}
\cr
p_{31} & p_{32} & p_{33} & p_{34} &
p_{35} & p_{36}
\cr
p_{41} & p_{42} & p_{43} & p_{44} &
p_{45} & p_{46}
\cr
p_{51} & p_{52} & p_{53} & p_{54} &
p_{55} & p_{56} } \rright]
\\
&&\qquad = \lleft[\matrix{ 0 &
a_{12} & a_{13}
\cr
0 & a_{22} & a_{23}
\cr
a_{31} & 0 & a_{33}
\cr
a_{41} & a_{42} & 0
\cr
a_{51} & a_{52} & a_{53} } \rright] \cdot
\lleft[\matrix{ 0 & b_{12} & b_{13} & b_{14} &
b_{15} & b_{16}
\cr
b_{21} & 0 & b_{23} & b_{24} & b_{25} &
b_{26}
\cr
b_{31} & b_{32} & 0 & b_{34} & b_{35} &
b_{36} } \rright].
\end{eqnarray*}
This parametrized variety has codimension $7$ and degree $735$ in
$\mathbb{P}^{29}$.
Its prime ideal is generated by $75$ quartics and $20$ sextics of the
desired row and column degrees.
\end{example}

The base case for Theorem~\ref{conjmingens} is
the case of $ 4 \times3$-matrices,
even though $\partial\mathcal{M} = \mathcal{M} \cap\Delta_{11}$ is trivial
in this case.

The corresponding ideal is principal,
and it is generated by the determinant of the $4 \times4$-matrix that is
obtained by deleting the fourth column of
(\ref{Concamatrix}).

The sextics in Theorem~\ref{conjmingens}
can be constructed as follows. Start with the polynomial
\[
\bigl( \bigl((a_1 \wedge a_2)\vee b_{1}
\bigr) \wedge a_3 \bigr) \vee\bigl( \bigl( (a_1 \wedge
a_2) \vee b_{2}\bigr) \wedge a_4 \bigr) \vee
b_{3} %
\]
that is given in (\ref{eqsixthree}). Now multiply this
with the $3 \times3$-minor $b_i \vee b_j \vee b_k$ of $B$.
The result has bidegree $(6,6)$ in the parameters
$(A,B)$ and can be written as a sextic
in $P = AB$. By construction, it vanishes on our component of
$\overline{\partial\mathcal{M}}$, and it has the
asserted degrees in the row and column gradings on $\mathbb{Q}[P]$.
This is the generator of the prime ideal referred to in Theorem~\ref
{conjmingens}.

Theorem~\ref{thmalgebraicboundary}
characterizes the probability distributions
in the algebraic boundary of our model,
but not those in the topological boundary, since
the following inclusion is strict:
%
\begin{equation}
\label{eqboundaryinclusion} \partial\mathcal{M}    \subset
\overline{\partial\mathcal
{M}} \cap\Delta_{mn-1}.
\end{equation}
In fact, the left-hand side is much smaller than the right-hand side.

To quantify the discrepancy between the two semialgebraic sets in
(\ref{eqboundaryinclusion}), we conducted the following experiment
in the smallest interesting case $m=n=4$.
We sampled from the component
(\ref{factorizationsofthefirsttype})
of $\overline{\partial\mathcal{M}} \cap\Delta_{15} $
by generating random rational numbers for
the nine
parameters $a_{ij}$ and the eight parameters
$b_{ij}$. This was done using the built-in
\texttt{Macaulay2} function \texttt{random(QQ)}\hspace*{-2pt}.
The resulting matrix in
$\overline{\partial\mathcal{M}} \cap\Delta_{15} $ was obtained
by dividing by the sum of the entries.
For each matrix, we tested whether it lies in
$\partial\mathcal{M}$. This was done using the criterion
in Corollary~\ref{cortopoboundary}.
The answer was affirmative only in $257$ cases out of $5000$ samples.
This suggests that $\partial\mathcal{M}$ occupies only a tiny part of the
set $ \overline{\partial\mathcal{M}} \cap\Delta_{15}$.
One of those rare points in the topological boundary is the matrix
%
\begin{eqnarray}
\label{eqnice44matrix} \lleft[\matrix{ 6 & 13 & 3 & 1
\cr
4 & 16 & 6 & 2
\cr
12 & 4 & 8 & 12
\cr
5 & 9 & 10 & 9
\cr
} \rright]   =   \lleft[\matrix{ 0 & 1 & 3
\cr
1 & 0 & 4
\cr
4 & 4 & 0
\cr
4 & 1 & 2 } \rright] \cdot\lleft[\matrix{ 0 & 0 & 2 & 2
\cr
3 & 1 & 0 & 1
\cr
1 & 4 & 1 & 0 } \rright].
\end{eqnarray}
To construct this particular example, the parameters $a_{ij}$ and
$b_{ij}$ were selected uniformly at random among the integers between
$1$ and $4$.
Only $1$ out of $1000$ samples gave a matrix lying in $\partial
\mathcal{M}$.
In fact, this matrix lies on precisely one of
the $304$ strata in the topological boundary
$\partial\mathcal{M}$.

We close this paper with a quantifier-free semialgebraic formula for
the topological boundary.

\begin{corollary} \label{cortopoboundary}
An $m \times n$-matrix $P $
lies on the topological boundary $\partial\mathcal{M}$ if and only if:
\begin{itemize}
\item the conditions of Theorem~\ref
{theoremsemialgebraicdescription} are satisfied, and

\item$P$ contains a zero, or $\rank(P)=3$ and for each $i,j,i',j'$
for which the conditions of Theorem~\ref
{theoremsemialgebraicdescription} are satisfied there exist $k,l$
such that {(\ref{eqsixthree})} $\cdot$ (\ref{eqsixthree})~$[k   \leftrightarrow  l]  =  0  $.
\end{itemize}
\end{corollary}

This corollary will be derived (in Appendix~\ref{appproofs}) from our
results in Section~\ref{sec3}.

\begin{appendix}
\section{Proofs}\label{appproofs}

This appendix furnishes the
proofs for all lemmas, propositions and corollaries
in this paper.

\begin{pf*}{Proof of Lemma~\ref{lemmaEMfixed}}
$(3) \Rightarrow(2)$:  If $(A, \Lambda, B)$ remains fixed after one
completion of the E-step and the M-step, then it will remain fixed
after any number of rounds of the E-step and the M-step.

$(2)\Rightarrow(3)$: By the proof of \cite{ASCB}, Theorem~1.15,
the log-likelihood function $\ell_U$ grows strictly after the
completion of an E-step and an M-step unless the parameters $(A,
\Lambda, B)$ stay fixed, in which case $\ell_U$ also stays fixed.
Thus, the only way to start with $(A, \Lambda, B)$ and to end with it
is for $(A, \Lambda, B)$ to stay fixed after every completion of an
E-step and an M-step.

$(2)\Rightarrow(1)$: If $(A, \Lambda, B)$ is the limit point
of EM when we start with it, then
it is in the set of all limit points.
This argument is reversible, and so we
also get $(1) \Rightarrow(2),(3)$.
\end{pf*}


\begin{pf*}{Proof of Lemma~\ref
{lemmageometricdescriptionofinteriorpoints}}
The if-direction of the first sentence follows from the following two
observations:
1.~The function that takes $P \in\R^{m\times n}_{\geq0}$ to the
vertices of $\mathcal{B}$ is continuous on all $m\times n$ nonnegative
matrices without zero columns, since the vertices of $\mathcal{B}$ are
of the form $P^j/P_{+j}$, where $P_{+j}$ denotes the $j$th column sum
of $P$. 2.~The function that takes $P \in\R^{m\times n}_{\geq0}$ to
the vertices of $\mathcal{A}$ is continuous on all $m\times n$
nonnegative matrices of rank $r$, since the vertices of $\mathcal{A}$
are solutions to a system of linear equations in the entries of $P$.

For the only-if-direction of the first sentence assume that $P$ lies in
the interior of~$\mathcal{M}_r$.
Each $P'$ of rank $r$ in a small neighborhood of $P$ has nonnegative
rank $r$. We can choose $P'$ in this neighborhood such that the columns
of $P'$ are in $\operatorname{span}(P)$ and $\cone(P')=t \cdot\cone(P)$
for some $t>1$.
Since $P'$ has nonnegative rank $r$, there exists an $(r-1)$-simplex
$\Delta$ such that $\mathcal{B}'\subseteq\Delta'\subseteq\mathcal
{A}$. Hence, $\mathcal{B}$ is contained in the interior of $\Delta'$.
Finally, the second sentence is the contrapositive of the first sentence.
\end{pf*}

\begin{pf*}{Proof of Corollary~\ref
{corollaryboundarypointsforrankthreematrices}}
The if-direction follows from the second sentence of Lemma~\ref
{lemmageometricdescriptionofinteriorpoints}. For the
only-if-direction, assume that $P\in\partial\mathcal{M}_3$ and it
contains no zeros.
We first consider the case $\operatorname{rank}(P) = 3$.
By Lemma~\ref{lemmageometricdescriptionofinteriorpoints}, every
triangle $\Delta$ with $\mathcal{B} \subseteq\Delta\subseteq
\mathcal{A}$ contains a vertex of $\mathcal{B}$ on its boundary.
Moreover, by the discussion above, every edge of $\Delta$ contains a
vertex of $\mathcal{B}$, and
(a) or (b) must hold. It remains to be seen that $\operatorname{rank}(P) \leq2
$ is impossible
on the strictly positive part of the boundary of $\mathcal{M}_3$.
Indeed, for
every rank $3$ matrix $P'$ in a neighborhood of $P$,
the polygons $\mathcal{A}', \mathcal{B}'$ have the
property that $\mathcal{B}'$ is very close to a line segment
strictly contained in the interior of $\mathcal{A}'$.
Hence, $t\mathcal{B}' \subseteq\Delta\subseteq\mathcal{A}'$
for some triangle $\Delta$.
Thus, $P' \notin\partial{\mathcal{M}}_3$ and, therefore,
$P \notin\partial{\mathcal{M}}_3$.
\end{pf*}

\begin{pf*}{Proof of Corollary~\ref
{corollarygeometricconditionsforamatrixtohavenonnegativerank3}}
The if-direction is immediate. For the only-if direction, consider any
$P \in\mathcal{M}_3$.
If $P\in\partial\mathcal{M}_3$, then the only-if-direction follows
from Corollary~\ref
{corollaryboundarypointsforrankthreematrices}. If $P$ lies in
the interior of $\mathcal{M}_3$, then let $t$ be maximal such
that $t\mathcal{B} \subseteq\Delta' \subseteq\mathcal{A}$ for some
triangle $\Delta'$. Then either a vertex of $\Delta'$ coincides with
a vertex of $\mathcal{A}$
or an edge of $\Delta'$ contains an edge of $t\mathcal{B}$. In the
first case, we
take $\Delta=\Delta'$. In the second case, we take $\Delta=\frac
{1}{t}\Delta'$.
In the first case, a vertex of $\Delta$ coincides with a vertex of
$\mathcal{A}$,
and in the second case, an edge of $\Delta$ contains an edge of
$\mathcal{B}$.
\end{pf*}

\begin{pf*}{Proof of Corollary~\ref{corollaryrationalfactorizations}}
If $P$ has a nonnegative factorization of size $3$, then it has one
that corresponds to a geometric condition in Corollary~\ref
{corollarygeometricconditionsforamatrixtohavenonnegativerank3}.
The left matrix in the factorization can be taken to be equal to the
vertices of the nested triangle, which can be expressed as rational
functions in the entries of $P$. Finally, the right matrix is obtained
from solving a system of linear equations with rational coefficients,
hence its entries are again rational functions in the entries of $P$.
\end{pf*}

\begin{pf*}{Proof of Proposition~\ref{propquivercycle}}
Consider the sequence of linear maps
%
\begin{equation}
\label{eqrnmr} \R^r    \buildrel{B^T}\over{
\longrightarrow}   \R^n    \buildrel{R}\over{\longrightarrow}
\R^m    \buildrel{A^T}\over{\longrightarrow}
\R^r.
\end{equation}
The ideal $\mathcal{C}$ says that the two compositions
are zero. It defines a \textit{variety of complexes}
\cite{CCA}, Example 17.8. The irreducible components
of that variety correspond to \textit{irreducible rank arrays} \cite{CCA},
Section~17.1,
that fit inside the format (\ref{eqrnmr}) and are maximal with this property.
By \cite{CCA}, Theorem 17.23, the quiver loci for these
rank arrays are irreducible
and their prime ideals are the ones we listed.
These can also be described by \textit{lacing diagrams} \cite{CCA},
Proposition~17.9.

The proof that $\mathcal{C}$ is radical was suggested to us by
Allen Knutson. Consider the Zelevinski map
\cite{CCA}, Section~17.2,
that sends the triple $(A^T, R, B^T)$ to the $(r+m+n+r) \times
(r+m+n+r) $ matrix
\begin{eqnarray*}
\lleft[\matrix{ 0 & 0 & B^T & 1
\cr
0 & R & 1 & 0
\cr
A^T & 1 & 0 & 0
\cr
1 & 0 & 0 & 0} \rright].
\end{eqnarray*}
Next, apply
the map that takes this matrix to the big cell (the open Borel orbit)
in the flag variety $GL(2r+m+n)/\operatorname{parabolic}(r, m, n, r)$
corresponding to the given block structure.

Our scheme is identified with the intersection of two Borel invariant
Schubert varieties. The first Schubert variety encodes the fact that
there are $0$'s in the North West block, and the
$(r+n+m)\times(r+m)$ North West rectangle has rank $\leq m$. The
second Schubert variety corresponds to the $(r+n)\times(r+m+n)$
North West rectangle having rank $\leq n$.
The intersection of Schubert varieties is reduced by \cite{BK},
Section~2.3.3, page~74.
Hence, the original scheme is reduced,
and we conclude that $\mathcal{C}$ is the
radical ideal defining the variety of complexes (\ref{eqrnmr}).
\end{pf*}

The following relations hold
for $P = AB$ and $R$ on the variety of critical points~$V(\mathcal{C})$:
%
\begin{equation}
\label{eqduality} P^T \cdot R   =   0 \quad\mbox{and}\quad R \cdot
P^T   =   0 .
\end{equation}
These bilinear equations characterize the
\textit{conormal variety} associated to a pair
of determinantal varieties. Suppose
$P$ is fixed and has rank $r$. Then $P$ is a nonsingular point in
$\mathcal{V}$, and (\ref{eqduality})
is the system of linear equations that
characterizes normal vectors $R$ to
$\mathcal{V}$ at~$P$.

\begin{example} \label{exC443Appendix}
Let $m=n=4$ and $r=3$.
Then $\mathcal{C}$ has four minimal primes,
corresponding to the four columns in the table below.
These are the ranks for generic points on that prime:
\begin{eqnarray*}
\operatorname{rank}(A) &=& 0,\qquad  \operatorname{rank}(A) = 1,\qquad  \operatorname{rank}(A) = 2,\qquad \operatorname{rank}(A) = 3,
\\
\operatorname{rank}(R) &=& 4,\qquad  \operatorname{rank}(R) = 3,\qquad  \operatorname{rank}(R) = 2,\qquad  \operatorname{rank}(R) = 1,
\\
\operatorname{rank}(B) &=& 0,\qquad  \operatorname{rank}(B) = 1,\qquad \operatorname{rank}(B) = 2\qquad  \operatorname{rank}(B) = 3.
\end{eqnarray*}
The lacing diagrams that describe these four irreducible components are
as follows:
\[
\mbox{
\includegraphics{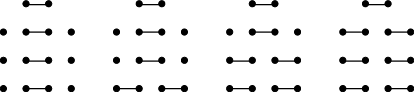}
}
\]
For instance, the second minimal prime is
$  \mathcal{C} + \langle 2 \times 2 $-minors of
$A$ and $B \rangle
+ \langle \operatorname{det}(R) \rangle
$.

Note that the ranks of $P = AB$
and $R$ are complementary on each irreducible component. They add up to~$4$.
The last component gives the behavior of EM for random data:
the MLE $P$ has rank~$3$,
it is a nonsingular point on the determinantal hypersurface
$\mathcal{V}$, and the normal space
at $P$ is spanned by the rank $1$ matrix $R$.
This is the duality (\ref{eqduality}).
The third component expresses the
behavior on the singular locus of $\mathcal{V}$.
Here, the typical rank of both $P$ and $R$ is $2$.
\end{example}

\begin{pf*}{Proof of Proposition~\ref{prop633}}
Let $f,g_1,g_2,g_3,g_4$ denote the
$4\times 4$ minors of the matrix (\ref{Concamatrix}), where $\operatorname{deg}(f) = 4$ and
$\operatorname{deg}(g_i) = 6$. Fix $i \in\{1,2,3,4\}$, select
$u_{11},\ldots,u_{44} \in\mathbb{N}$ randomly, and set
%
\begin{eqnarray}
L  =  \lleft[\matrix{ u_{11} & u_{12} & \cdots&
u_{44}
\cr
p_{11} & p_{12} & \cdots& p_{44}
\cr
p_{11} \, \partial f   / \partial p_{11} & p_{12}
\,  \partial f   / \partial p_{12} & \cdots& p_{44}
\,\partial f   / \partial p_{44}
\cr
p_{11} \, \partial g_i / \partial p_{11} &
p_{12} \,  \partial g_i / \partial p_{12} &
\cdots& p_{44} \,  \partial g_i / \partial
p_{44} } \rright].
\end{eqnarray}
This is a $4 \times 16$ matrix.
Let $\lambda_1$ and $ \lambda_2$ be new unknowns and consider the
row vector
%
\begin{equation}
\label{eqLL} \lleft[\matrix{   1 & -u_+ & \lambda_1 &
\lambda_2   } \rright] \cdot L.
\end{equation}
Inside the polynomial ring $\mathbb{Q}[p_{ij}, \lambda_k]$
with $20$ unknowns,
let $I$ denote the ideal generated by
$\{f,g_1,g_2,g_3,g_4\}$, the $16$ entries of (\ref{eqLL}),
and the linear polynomial
$p_{11} + p_{12} + \cdots+ p_{44} - 1$.
Thus, $I$ is the ideal of \textit{Lagrange likelihood equations}
introduced in \cite{GR}, Definition~2.
Gross and Rodriguez~\cite{GR}, Proposition~3,
showed that $I$ is a $0$-dimensional radical ideal,
and its number of roots is the ML degree of the variety
$V(f,g_1,g_2,g_3,g_4)$. We computed a
Gr\"obner bases for $I$ using the
computer algebra software \texttt{Magma}~\cite{magma}.
This computation reveals that $V(I)$ consists of $633$ points over
$\mathbb{C}$.
\end{pf*}

\begin{pf*}{Proof of Corollary~\ref{cortopoboundary}}
A matrix $P$ has nonnegative rank $3$ if and only if the conditions of
Theorem~\ref{theoremsemialgebraicdescription} are satisfied. Assume
$\rank(P)=3$. By Corollary~\ref
{corollaryboundarypointsforrankthreematrices}, a matrix $P\in
\mathcal{M}$ lies on the boundary of $\mathcal{M}$ if and only if it
contains a zero or for any triangle $\Delta$ with $\mathcal{B}
\subseteq\Delta\subseteq\mathcal{A}$ every edge of $\Delta$
contains a vertex of $\mathcal{B}$ and (a) or (b) holds. By proof of
Theorem~\ref{theoremsemialgebraicdescription}, the latter implies
that for each $i,j,i',j'$ for which the conditions of Theorem~\ref
{theoremsemialgebraicdescription} are satisfied there exist $k,l$
such that {(\ref{eqsixthree})} $\cdot$
{(\ref{eqsixthree})}\,$[k   \leftrightarrow  l]=0$. On the other
hand, if $P$ lies in the interior of $\mathcal{M}^{m \times n}_3$,
then by the proof of Corollary~\ref
{corollarygeometricconditionsforamatrixtohavenonnegativerank3},
the following holds:
there exists a triangle $\Delta$ with a vertex coinciding
with a vertex of $\mathcal{A}$ or
with an edge containing an edge of $\mathcal{B}$,
and such that the inequality {(\ref{eqsixthree})} $\cdot$
{(\ref{eqsixthree})}\,$[k   \leftrightarrow  l]>0$
holds for all $k,l$ in the corresponding semialgebraic condition.
\end{pf*}\vspace*{-8pt}

\section{Basic concepts in algebraic geometry}\label{appbasics}
This appendix gives a synopsis of basic
concepts from algebraic geometry that are used in this paper.
It furnishes the language to speak about
solutions to polynomial equations in many variables.

\subsection{Ideals and varieties}
Let $R = K[x_1,\ldots,x_n]$ be
the ring of polynomials in $n$ variables with coefficients in
a subfield $K$ of the real numbers $\mathbb{R}$,
usually the rational numbers $K = \mathbb{Q}$.
The concept of an ideal $I$ in the ring $R$ is similar to the concept
of a normal subgroup
in a group.

\begin{definition} A subset $I\subseteq R$ is an \textit{ideal} in $R$ if
$I$ is an subgroup of $R$ under addition, and
for every $f\in I$ and every $g\in R$ we have
$ fg \in I$.
Equivalently, an ideal $I$ is closed under taking linear combinations
with coefficients in the ring~$R$.
\end{definition}

Let $T $ be any set of polynomials in $R$. Their set of zeros
is called the \textit{variety} of~$T$. It is denoted
\[
V(T)   =    \bigl\{P\in\mathbb C^n \dvtx  f(P) = 0\mbox{ for all
}f\in T \bigr\}.
\]
Here, we allow zeros with complex coordinates. This greatly simplifies
the study of $V(T)$ because $\mathbb{C}$ is algebraically closed,
\textit{that is}, every nonconstant polynomial has a zero.

The \textit{ideal generated by} $T$, denoted by $\langle T\rangle$,
is the smallest ideal in $R$ containing~$T$. Note that
\[
V(T)   =   V\bigl(\langle T\rangle\bigr). %
\]
In computational algebra, it is often desirable to replace
the given set $T$ by a \textit{Gr\"obner basis} of $\langle T \rangle$.
This allows us to test ideal membership and to determine
geometric properties of the variety $V(T)$.

\begin{definition} A subset $X\subseteq\mathbb C^n$ is a \textit{variety}
if $X=V(T)$ for some subset $T \subseteq R$.
\end{definition}

Hilbert's basis theorem ensures that here $T$ can always be chosen
to be a finite set of polynomials.
The concept of variety allows us to define a new topology on~$\mathbb C^n$.
It is coarser than the usual topology.

\begin{definition} We define the \textit{Zariski topology} on $\mathbb
C^n$ by taking closed sets to be the varieties and open sets to be the
complements of varieties.
This topology depends on the choice of $K$.
\end{definition}

If $K = \mathbb{Q,}$ then
$ X = \{+\sqrt{2}, -\sqrt{2} \}$ is a variety
(for $n=1$) but
$Y = \{+\sqrt{2} \}$ is not a variety.
Indeed, $X = \overline{Y}$ is the \textit{Zariski closure} of $Y$,
\textit{that is}, it is the smallest variety containing $Y$,
because the minimal polynomial of $\sqrt2$ over $\mathbb Q$ is $x^2 - 2$.
Likewise, the set of $1618$ points in Example~\ref{exgreencurve}
is a variety in $\mathbb{C}^2$. It is the Zariski closure of the four
points on the topological boundary on the left in Figure~\ref{figgreencurve}.
The following proposition justifies the fact that the Zariski topology
is a topology.

\begin{proposition} Varieties satisfy the following properties:
\begin{longlist}[3.]
\item[1.] The empty set
$\varnothing= V(R)$ and the whole space $\mathbb C^n = V(\langle
0\rangle)$ are varieties.
\item[2.] The union of two varieties is a variety:
\[
V(I)\cup V(J)   =    V(I \cdot J )   =    V( I \cap J).
\]
\item[3.] The intersection of any family of varieties is a variety:
\[
\bigcap_{i\in\mathcal I} V(I_i)   =   V\bigl(
\langle I_i \dvtx  i\in\mathcal I\rangle\bigr).
\]
\end{longlist}
\end{proposition}

Given any subset $X\subseteq\mathbb C^n$ (not necessarily a variety),
we define the \textit{ideal} of $X$ by
\[
I(X)    =    \bigl\{f\in R \dvtx  f(P) = 0\mbox{ for all }P\in X\bigr\}.
\]
Thus, $I(X)$ consists of all polynomials in $R$ that vanish on $X$.
The \textit{Zariski closure}
$\overline X$ of $X$ equals
\[
\overline X   =   V\bigl(I(X)\bigr). %
\]

\subsection{Irreducible decomposition}
A variety $X\subseteq\mathbb C^n$ is \textit{irreducible} if we cannot
write $X = X_1\cup X_2$, where $X_1, X_2\subsetneq X$ are strictly
smaller varieties. An ideal $I\subseteq R$ is \textit{prime} if $fg\in I$
implies $f\in I$ or $g\in I$.
For instance, $ I (\{ \pm\sqrt{2} \}) = \langle x^2-2 \rangle$ is a
prime ideal in $\mathbb{Q}[x]$.

\begin{proposition} The variety $X$ is irreducible if and only if
$I(X)$ is a prime ideal.
\end{proposition}

An ideal is \textit{radical} if it is an intersection of prime ideals.
The assignment $X \mapsto I(X) $ is a bijection between
varieties in $\mathbb{C}^n$ and radical ideals in $R$.
Indeed, every variety $X$ satisfies $V(I(X)) = X$.

\begin{proposition} Every variety $ X$ can be written uniquely as
$X = X_1\cup X_2\cup\cdots\cup X_m$,
where $X_1, X_2, \ldots, X_m$ are irreducible and none of these $m$ components
contains any other. Moreover,
\[
I(X)   =    I(X_1) \cap I(X_2) \cap\cdots\cap
I(X_m) %
\]
is the unique
decomposition of the radical ideal $I(X)$ as an intersection of prime ideals.
\end{proposition}

For an explicit example, with $m=11$, we consider the ideal
(\ref{primarydecompositionmix22}) with the last
intersectand removed. In that example, the
EM fixed variety $X$ is decomposed into
$11$ irreducible components.

All ideals $I$ in $R$
can be written as intersections of \textit{primary ideals}.
Primary ideals are more general than
prime ideals, but they still define irreducible
varieties. A \textit{minimal prime} of an ideal $I$
is a prime ideal $J$ such that $V(J)$
is an irreducible component of $V(I)$.
See \cite{Solving}, Chapter~5,
for the basics on \textit{primary decomposition}.

\begin{definition}
Let $I \subseteq R$ be an ideal and $f \in R$ a polynomial. The
saturation of $I$ with respect to $f$ is the ideal
\[
\bigl(I\dvtx f^{\infty}\bigr)=\bigl\langle g \in R\dvtx  gf^k \in I
\mbox{ for some } k>0 \bigr\rangle. %
\]
\end{definition}

Saturating an ideal $I$ by a polynomial $f$ geometrically means that we
obtain a new ideal $J = (I\dvtx f^{\infty})$ whose variety $V(J)$ contains
all components of the variety $V(I)$ except for the ones on which $f$ vanishes.
For the more on these concepts from algebraic geometry we recommend
the text \cite{CLO}.

\subsection{Semialgebraic sets}
The discussion above also applies if
we consider the varieties $V(T)$ as subsets of $\mathbb R^n$ instead of
$\mathbb C^n$. This brings us to the world of
\textit{real algebraic geometry}. The field $\mathbb R$ of real numbers is
not algebraically
closed, it comes with a natural order, and it is fundamental for applications.
These features explain why real algebraic geometry is a subject in its
own right.
In addition to the polynomial equations we discussed so far,
we can now also introduce inequalities.

\begin{definition}
A \textit{basic semialgebraic set} $X\subseteq\mathbb R^n$ is a subset of
the form
\[
X   =   \bigl\{P\in\mathbb R^n\dvtx   f(P) = 0 \mbox{ for all }f\in T
\mbox{ and } g(P)\geq0 \mbox{ for all }g\in S \bigr\},
\]
where $S$ and $T$ are finite subsets of $R$.
A \textit{semialgebraic set} is a subset $X\subseteq\mathbb R^n$
that is obtained by a finite sequence of unions, intersections and
complements of basic semialgebraic sets.
\end{definition}

In other words, semialgebraic sets are described by finite
Boolean combinations of polynomial equalities and polynomial inequalities.
For basic semialgebraic sets, only conjunctions are allowed.
For example, the following two simple subsets of the plane are both
semialgebraic:
\begin{eqnarray*}
X &=& \bigl\{(x,y) \in\R^2\dvtx  x \geq0   \mbox{ and }   y \geq0 \bigr\}
\quad\mbox{and}
\\
Y &=& \bigl\{(x,y) \in\R^2\dvtx  x \geq0   \mbox{ or }
 y \geq0 \bigr\} .
\end{eqnarray*}
The set $X$ is basic semialgebraic, but $Y$ is not.
All convex polyhedra are semialgebraic.
A fundamental theorem due to Tarski states that the image of
a semialgebraic set under a polynomial map is semialgebraic.
Applying this to the map (\ref{eqmapphi}), we see that
the model $\mathcal{M}$ is semialgebraic.
The boundary of any semialgebraic set is again semialgebraic. The formulas
in Theorem~\ref{theoremsemialgebraicdescription}
and Corollary~\ref{cortopoboundary} make this explicit.
For more on semialgebraic sets and real algebraic geometry, see~\cite{BPR}.
\end{appendix}

\section*{Acknowledgements}
This work was carried out at the Max-Planck-Institut
f\"ur Mathematik in Bonn,
where all three authors were based during the Fall of 2013.

We thank Aldo Conca, Allen Knutson, Pierre-Jean Spaenlehauer and Matteo Varbaro
for helping us with this project.
Mathias Drton, Sonja Petrovi\'c, John Rhodes, Caroline Uhler and Piotr Zwiernik
provided comments on various drafts of the paper. We thank
Christopher Miller for pointing out an inaccuracy in Example~\ref{exgreencurve}.




\printaddresses
\end{document}